\documentclass[11pt]{article}
 
\usepackage[margin = 1in]{geometry}
\usepackage{amsfonts,latexsym,amsmath,amssymb,amsthm}
\usepackage[mathscr]{eucal}
\usepackage{graphicx,color}
\usepackage{comment}
\usepackage{hyperref}
\usepackage{epstopdf}

\usepackage{mathrsfs} 

\newtheorem{theorem}{Theorem}
\newtheorem{lemma}{Lemma}
\newtheorem{proposition}{Proposition}

\newtheorem{definition}{Definition}
\newtheorem{assumption}{Assumption}
\newtheorem{remark}{Remark}

\newtheorem{problem}{Problem}
\newtheorem{example}{Example}

\newcommand{\N}{\ensuremath{\mathbb{N}}}
\newcommand{\R}{\ensuremath{\mathbb{R}}}
\newcommand{\NN}{\ensuremath{\mathbb{N}}}
\newcommand{\RR}{\ensuremath{\mathbb{R}}}
\newcommand{\regparam}{\ensuremath{\beta}}
\newcommand{\argmin}{\mathrm{argmin}}

\renewcommand{\d}{\,\mathrm{d}}
\newcommand{\Leb}{\mathscr{L}} 

\newcommand{\spt}[1]{\mathrm{supp}(#1)}

\newcommand{\genvar}{\mathbf{z}}

\newcommand{\X}{X}
\newcommand{\Y}{Y}
\newcommand{\xx}{\mathbf{x}}
\newcommand{\pp}{\mathbf{p}}

\def\Mmat{\mathbf{M}}
\def\Imat{\mathbf{I}}

\def\sign{\mathrm{sign}}

\usepackage[normalem]{ulem}
\usepackage{color}

\begin{document}

\title{ \LARGE \bf%
	Semi-algebraic approximation using Christoffel-Darboux kernel
}
\author{Swann Marx$^1$, Edouard Pauwels$^2$, Tillmann Weisser$^3$, Didier Henrion$^{4,5}$\\
 and Jean Bernard Lasserre$^{4,6}$}

\footnotetext[1]{LS2N, \'Ecole Centrale de Nantes $\&$ CNRS UMR 6004, F-44000 Nantes, France.}
\footnotetext[2]{IRIT-UPS, Universit\'e de Toulouse, 118 route de Narbonne, F-31400 Toulouse, France.}
\footnotetext[3]{Theoretical Division and Center for Nonlinear Studies, Los Alamos National Laboratory, Los Alamos NM 87545, USA}
\footnotetext[4]{LAAS-CNRS, Universit\'e de Toulouse, 7 avenue du colonel Roche, F-31400 Toulouse; France}
\footnotetext[5]{Faculty of Electrical Engineering, Czech Technical University in Prague, Technick\'a 4, CZ-16206 Prague, Czechia}
\footnotetext[6]{IMT-UPS, Universit\'e de Toulouse, 118 Route de Narbonne, F-31400 Toulouse, France}

\maketitle

\abstract{We provide a new method to approximate a (possibly discontinuous) function using Christoffel-Darboux kernels. Our knowledge about the unknown multivariate function is in terms of finitely many moments of the Young measure supported on the graph of the function. Such an input is available when approximating weak (or measure-valued) solution of optimal control problems, entropy solutions to non-linear hyperbolic PDEs, or using numerical integration from finitely many evaluations of the function. While most of the existing methods construct a piecewise polynomial approximation, we construct a semi-algebraic approximation whose estimation and evaluation can be performed efficiently. An appealing feature of this method is that it deals with non-smoothness implicitly so that a single scheme can be used to treat smooth or non-smooth functions without any prior knowledge. On the theoretical side, we prove pointwise convergence almost everywhere as well as convergence in the Lebesgue one norm under broad assumptions. Using more restrictive assumptions, we obtain explicit convergence rates. We illustrate our approach on various examples from control and approximation.
In particular we observe empirically that our method does not suffer from the the Gibbs phenomenon when approximating discontinuous functions.  
\\[1em]
	}\\[1em]
\newpage

\tableofcontents
\newpage

\section{Introduction}\label{sec:intro}
In this paper we address the following generic inverse problem.
Let
\[
\begin{array}{rclcl}
f & : & \X & \to & \Y \\
&& \xx:=(x_1,\:x_2,\ldots,x_{p-1}) & \mapsto & y
\end{array}
\]
be a bounded measurable function from a given compact set $\X \subset\RR^{p-1}$ to a given compact set $\Y \subset\RR$, with $p\geq 2$. We assume that $\X$ is equal to the closure of its interior. 

Given $d \in \N$, consider a vector of polynomials of total degree at most $d$
\[
\mathbf{b}(\xx,y):\: (\xx,\:y) \in\mathbb{R}^p\mapsto \left(b_1(\xx,\:y) \quad b_2(\xx,\:y) \quad\cdots\quad b_{n_d}(\xx,\:y)\right)
\in\mathbb{R}^{n_d},
\]
 where $n_d:=\begin{pmatrix}
p+d\\d
\end{pmatrix}$ which is understood as the binomial coefficient. 
For example, $\mathbf{b}$ may be a vector whose entries are the polynomials of the canonical monomial basis or any orthonormal polynomial basis, e.g. Chebyshev or Legendre. Associated to $\mathbf{b}$ and $f$, let
\[
\int_{\X} \mathbf{b}(\xx,f(\xx))\mathbf{b}(\xx,f(\xx))^\top d\xx
\]
be the moment matrix of degree $2d$, where the integral is understood entry-wise.
 
\begin{problem}[Graph inference from moment matrix]
\label{eq:mainProblem}
Given the moment matrix of degree $2d$, compute an approximation $f_d$ of the function $f$, with convergence guarantees when degree $d$ tends to infinity.
\end{problem}

\subsection{Motivation} 

Inverse Problem \ref{eq:mainProblem} is encountered in several interesting situations.
In the weak (or measure-valued) formulation of Optimal Control problems (OCP) \cite{vinter1993convex,hernandez1996linear,lasserre2008nonlinear}, Markov Decision Processes \cite{jbl-ohl}, option pricing in finance \cite{finance}, stochastic control and optimal stopping \cite{stockbridge},
and some non-linear partial differential equations (PDEs) \cite{diperna1985measure}, non-linear non-convex problems are formulated as linear programming (LP) problems on occupation measures. Instead of the classical solution, the object of interest is a measure supported on the graph of the solution. Numerically, we optimize over finitely many moments of this measure.

Following the notation introduced above for stating Problem \ref{eq:mainProblem},
and letting
\[
\genvar:=(\xx,\:y) \in \R^p,
\]
the moment matrix of degree $2d$ reads
\[
\Mmat_{\mu,d} := \int \mathbf{b}(\genvar)\mathbf{b}(\genvar)^\top d\mu(\genvar)
\]
corresponding to the measure
\begin{equation}\label{measure}
d\mu(\genvar):=\mathbb{I}_\X(\xx)\,d\xx\,\delta_{f(\xx)}(dy)
\end{equation}
supported on the graph
\[
\{(\xx,f(\xx)):\xx\in\X\} \subset X\times Y
\]
of function $f$, where $\mathbb{I}_{\X}$ denotes the indicator function of $\X$ which takes value $1$ on $\X$ and $0$ otherwise, and $\delta_{f(\xx)}$ denotes the Dirac measure at $f(\xx)$.

For instance, in OCPs an optimal occupation measure $\mu$ is supported on the graphs of optimal state-control trajectories. In order to recover a particular state resp. control trajectory it suffices to consider the moments of the marginal of the occupation measure $\mu$ with respect to time-state resp. time-control. Then with our notation, $\xx$ is time and $y$ is a state resp. control coordinate. Similarly, for the measure-valued formulation of non-linear first-order scalar hyperbolic PDEs, an occupation measure is supported 
on the graph of the unique optimal entropy solution. Then with our notation, $\xx$ is time and space and $y$ is the solution. From the knowledge of the moments of the occupation measure, we want to approximate the solution.  The measure $\mu$ can be disintegrated into its marginal on $\X$ and its conditional on $\Y$ given $\mathbf{x}$ in $\X$. The latter is also called a parametrized measure or a Young measure, see e.g. \cite{fattorini1999infinite}. 

This weak formulation has been used in a number of different contexts to prove existence and sometimes uniqueness of solutions. It turns out that it can also be used for {effective computation} as it fits perfectly the LP-based methodology described in \cite{stockbridge2} and the Moment-SOS (polynomial sums of squares) methodology described in e.g. \cite{lasserre2008nonlinear,lasserre2009moments,lasserre2019moments}. In the latter methodology one may thus approximate the optimal solution $\mu$ of the measure-valued formulation by solving  a hierarchy of semidefinite relaxations of the problem, whose size increases with $d$; see e.g. \cite{lasserre2008nonlinear} for OCPs and \cite{marx2018moment} for non-linear PDEs. An optimal solution of each semidefinite relaxation is a finite matrix of pseudo moments (of degree at most $2d$) which approximate those of $\mu$. This approach allows to approximate values for the corresponding variational problems but it does not provide any information about the underlying minimizing solutions beyond moments of measures supported on these solutions. Therefore an important and challenging practical issue consists of recovering from these moments an approximation of the trajectories of the OCP or PDEs. This is precisely an instance of Problem \ref{eq:mainProblem}.

 Another potential  target application of our method is the optimal
transportation problem, see e.g. \cite[Chapter 1]{S15} and references
therein.  In its original formulation by Monge, it is a highly nonlinear nonconvex optimization problem. Its relaxation by Kantorovich is a
linear optimization problem on measures, and hence on moments if the
data are semialgebraic. Under convexity assumptions, this linear problem
  has a unique measure solution called optimal transportation plan, supported on the graph of a function called the optimal transportation map.  Our method can be used to approximate separately each coordinate of the transportation map by only considering the submatrix of moments associated with a suitable marginal, extracted from an optimal solution of the semidefinite relaxation (which considers \emph{all} pseudo-moments up to a given degree). In view of the form of $\mu$, one still obtains the required convergence guarantees (e.g. pointwise), under appropriate assumptions described in the paper.

More generally, moment information about the unknown function $f$ in the format of
Problem \ref{eq:mainProblem} is available when applying the Moment-SOS hierarchy \cite{lasserre2019moments} to solve Generalized Moment Problems where the involved Borel measures are Young  measures. The necessary moment information is also given when considering empirical measures \cite{pauwels2016sorting,lasserre2017empirical,pauwels2018spectral} if input data points lie on the graph of an unknown function $f$ (e.g., as is the case in interpolation). On the other hand, in some other applications like image processing or shape reconstruction, moment information is available only for the measure $f(\xx)\:d\xx$, i.e. $y \mapsto \mathbf{b}(\xx,y)$ is linear.  Finally, our method and results would apply to measures $\nu$ supported on the graph of $f$, different from $\mu$ in \eqref{measure}. Provided that $\nu$ and $\mu$ are mutually absolutely continuous with bounded densities, we would recover similar convergence results modulo constants. 

\subsection{Contribution}

We address Problem \ref{eq:mainProblem} by providing an algorithm to approximate a (possibly discontinuous) unknown function $f$ from the moment matrix of the measure \eqref{measure} supported on its graph. The approximation converges to $f$ (in a suitable sense described later on) when the number of moments tends to infinity.

\paragraph{Proposed approximation scheme}

As is well-known in approximation theory, the sequence of Christoffel-Darboux polynomials associated with a measure is an appropriate tool to approximate accurately the support of the measure, hence the graph of $f$ in our case. Christoffel-Darboux kernels and functions are closely related to orthogonal polynomials \cite{szeg1939orthogonal}, \cite{dunkl2014orthogonal} and approximation theory \cite{nevai1986geza}, \cite{de2014multivariate}. Their asymptotic behavior (i.e., when the degree goes to infinity) provides useful and even crucial information on the support and the density of the underlying measure. A quantitative analysis is provided in \cite{mate1980bernstein}, \cite{gustafsson2009bergman} for single dimensional problem and in \cite{kroo2012christoffel} in a multivariate setting. Even more recently, in \cite{pauwels2016sorting}, \cite{lasserre2017empirical} and \cite{pauwels2018spectral}, Christoffel-Darboux polynomials have been used to approximate the support of Borel measures in a multivariate setting in the context of Machine Learning and Data Science. 

We propose a simple scheme to approximate the graph of $f$ based on the knowledge of the moment matrix of degree $2d$ of $\mu$. To that end we first compute the Christoffel-Darboux polynomial using a spectral decomposition of the moment matrix. The Christoffel-Darboux polynomial is an SOS polynomial $q_d(\xx,y)$ of degree $2d$ in $p$ variables. For every fixed $\xx$ we define
\[
f_d(\xx):=\argmin_{y\in\Y} q_d(\xx,y)
\]
which is a semi-algebraic function\footnote{A semi-algebraic function is a function whose graph is semi-algebraic, i.e. defined with finitely many polynomial inequalities.}, assuming for the moment for the ease of exposition that the above argmin is uniquely defined. This class of functions is quite large. For example, all polynomials of degree at most $d$ can be expressed using this technique: let $r$ be a polynomial in $\xx$ of degree $d$, then $q:(\xx,y)\mapsto (r(\xx)-y)^2$ is a degree $2d$ SOS polynomial whose partial minimization in $y$ yields $y = r(\xx)$ for all $\xx$. However, this class contains many more functions, including non-smooth semi-algebraic functions such as signs or absolute values. In particular this class of functions can be used to describe efficiently discontinuous functions, a typical case encountered in e.g. OCP problems with bang-bang controls and solutions with shocks for non-linear PDEs.

\begin{example}[Sign function as SOS partial minimum]

    \begin{figure}
        \centering
        \includegraphics[width=0.7\textwidth]{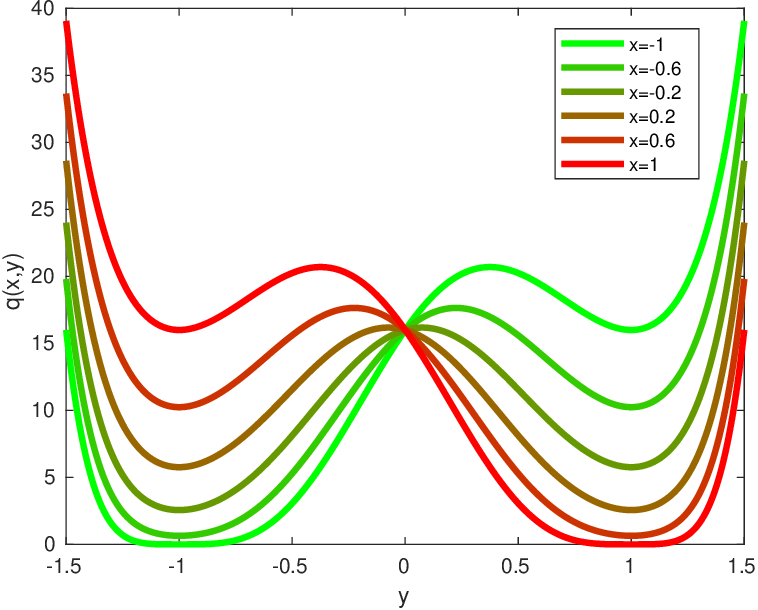}
        \caption{SOS polynomial $q(\xx,y)$ whose argmin in $y$ is the sign of $\xx$ on $[-1,1]$.}
        \label{fig:illustrsemi-algebraic}
    \end{figure}
    \label{ex:stepAbsVal}

    Consider the polynomial
    \begin{align}
        p_1 \colon\qquad \RR^2 &\mapsto \RR \nonumber\\
        (\xx,y) &\mapsto 4 - 3 \xx y - 4 y^2 + \xx y^3 + 2 y^4.
        \label{eq:illustrsemi-algebraic}
    \end{align}
    One can easily check that
\[
\argmin_{y \in \Y} p_1(\xx,y) = \sign(\xx) := \left\lbrace\begin{array}{ll}
-1 & \text{if}\quad \xx<0 \\
\{-1,1\} & \text{if}\quad  \xx=0 \\
1 & \text{if}\quad \xx>0 \\
\end{array}\right.
\]
for any $\xx \in \X:=[-1, 1]$ and $\Y\subset\RR$.
Note that since $p_1(\xx,\cdot)$ is positive for $\xx \in [-1,1]$, it can be squared without changing the argmin in $y$ and hence we obtain a similar representation of the sign function in the form of partial minimization in $y$ of the degree 8 SOS polynomial $q(\xx,y):=p_1^2(\xx,y)$, represented in Figure \ref{fig:illustrsemi-algebraic}.
\end{example}

\begin{example}[Absolute value as SOS partial minimum]
The reader may check that the argmin in $y$ of the (square of the positive) polynomial $11 - 12\xx^4 y - 6 \xx^2 y^2 + 4\xx^2y^3 +  3 y^4$ is equal to $|\xx|$ for all $\xx \in [-1,1]$.
\end{example}

To the best of our knowledge, this work is the first contribution where this class of semi-algebraic functions is used for graph approximations. The present work shows how these functions may be used to approximate discontinuous functions accurately.

\paragraph{Comparison to existing approximation approaches:}
A classical alternative to our approach would be to use 
$\Leb^2$-norm Legendre or Chebyshev approximations of the function $f$ which are also based on moment information. However these approaches only use moment information on the measure $f(\xx)\:d\xx$, i.e. $y \mapsto \mathbf{b}(\xx,y)$ is linear. Moreover, the support of this measure is not the graph of $f$. 

We claim that the full moment information provides useful additional data on the graph of $f$ which we can access using Christoffel-Darboux kernels associated with the measure $\mu$ in \eqref{measure}.
Note that in interpolation applications, we have the possibility to estimate the higher order moments of $f$ from finitely many evaluations of $f$ through Riemann integral approximations or Monte-Carlo approximations for example. However, in situations where we have neither access to higher moments nor pointwise evaluation of $f$, our method cannot be applied; 
signal processing applications are a typical case of the latter situations.

In general, approximating a discontinuous function $f$ is a challenge. Indeed, most well-known techniques suffer from the {Gibbs phenomenon}, i.e. the approximation produces oscillations at each point of discontinuity of $f$, see e.g. \cite{gottlieb1997gibbs} for a good survey on this topic. The main tools usually rely on properties of orthogonal polynomials \cite{szeg1939orthogonal} and the resulting approximations are based on a finite number of Fourier coefficients of the latter functions, i.e., typically first degree moment information on $f$. Projecting a discontinuous function on a class of infinitely smooth functions is the typical mechanism producing Gibbs phenomenon. In order to get rid of such a curse, {additional} techniques and prior information is needed. A possible approach is reported in \cite{eckhoff1993accurate} in the univariate case ($p=2$ in our notations), where a good approximation of locations of discontinuity points and jump magnitude is obtained by solving an appropriate (univariate) polynomial equation. Recent developments have effectively shown that such approaches may tame the Gibbs phenomenon \cite{batenkov2012algebraic,batenkov2015complete}. Iterative numerical methods can also be used to identify the points of discontinuities of $f$ (and of its derivatives) so as to construct accurate approximations locally in each identified interval, see e.g. \cite{pachon2010piecewise} in the case of Chebyshev polynomials. However such ad-hoc techniques are very specific to the univariate setting. 

On the contrary, our approach is not based on projection on a subspace of smooth functions, or identification of points of discontinuities of the function to be approximated. It is based on geometric approximation of the support of a singular measure using semi-algebraic techniques. An important feature of this approach is that the resulting approximant functions are not necessarily smooth, and furthermore, discontinuities (if any) can be treated implicitly only based on the whole moment information. As a result, we obtain a single approximation scheme, which (i) may adapt to the smoothness features of the target function $f$ without requiring prior knowledge of it, and (ii) can be used for multivariate $f$, both points being important challenges regarding numerical approximation.

\paragraph{Description of our contribution}
\begin{itemize}
\item[1.] We first provide a numerical scheme which allows to approximate the compact support of a measure which is singular with respect to the Lebesgue measure. We need to adapt the strategy of \cite{lasserre2017empirical} which covered the absolutely continuous case. This result may be considered of independent interest and will be instrumental to providing convergence guarantees for our approach.
\item[2.] Next, given a degree $d \in \NN$, we provide an approximation 
$f_d$ of the function $f$ and prove that the sequence $(f_d)_{d\in \NN}$ converges pointwise to $f$ almost everywhere as well as in $\Leb^1$-norm as $d$ goes to infinity (under broad assumptions on $f$). Furthermore, if we assume more regularity on $f$, then we also provide estimates for the rate of convergence. More precisely, we obtain $O(d^{-1/2})$ for multivariate Lipschitz functions and $O(d^{-1/4})$ for univariate functions of bounded variation.
\item[3.] Finally, we provide some numerical examples to illustrate the efficiency of the method. We first use our algorithm to approximate discontinuous or non-smooth solutions of OCP or PDE problems based on the Moment-SOS hierarchy. These experiments empirically demonstrate the absence of Gibbs phenomena.
We also provide an example where only samples of the measure under consideration are given and show that our algorithm also works well, even for moderate size samples, showing that moment input data could in principle be approximated using numerical integration methods.
\end{itemize}
\begin{example}[Sign function]
To give a flavor of what can be obtained numerically, consider the measure \eqref{measure} supported on the graph $\{(\xx,f(\xx)):\xx\in \X\}\subset \X\times\Y$ of the function $\xx\mapsto f(\xx):= \sign(\xx)$, with $\X:=\Y:=[-1,1]$. In Figure \ref{fig:ex1} (right) the resulting approximation $f_2$ with a moment matrix of size $6$ and degree $4$ (i.e. $15$ moments) cannot be distinguished from $f$. On the other hand, on the left, their Chebsyhev interpolants of degrees 4 and 20 (obtained with {\tt chebfun} \cite{chebfun}) illustrate the typical Gibbs phenomenon, namely oscillations near the discontinuity points that cannot be attenuated by increasing the degree. This phenomenon can be reduced or suppressed by identifying the discontinuity points and splitting $\X$ into intervals (as in e.g. \cite{pachon2010piecewise} and also implemented in {\tt chebfun}), but this strategy works only in the univariate case. In contrast, our algorithm does not attempt to approximate directly a univariate function with one or several univariate polynomials of increasing degree, but with the argmin of a bivariate polynomial of increasing degree. Moreover, our algorithm works also for multivariate functions, as shown by numerical examples later on.

 \begin{figure}
        \centering
        \includegraphics[width=.45\textwidth]{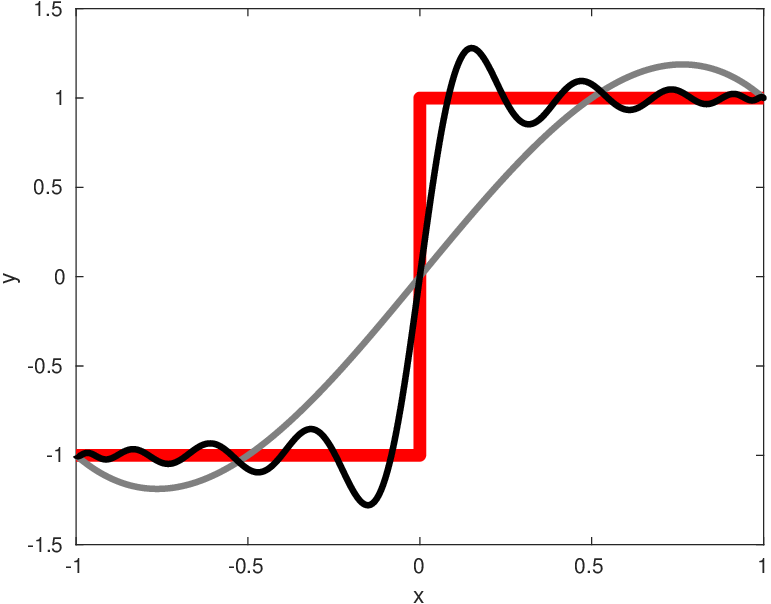}\includegraphics[width=.45\textwidth]{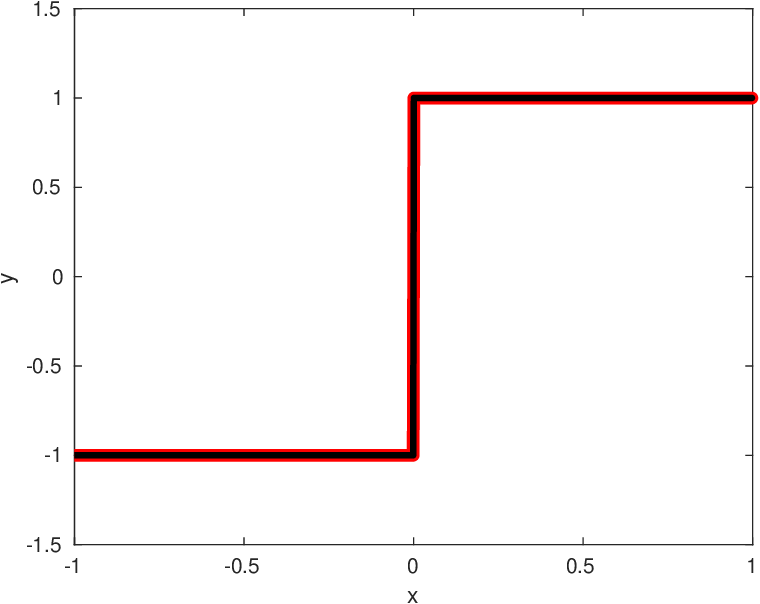}
        \caption{On the left, Chebyshev interpolant of degrees 4 (gray) and 20 (black) of the step function (red), featuring the typical Gibbs phenomenon. On the right, the proposed semi-algebraic approximation of degree 4 (black) of the same step function (red). The approximation cannot be distinguished from the step function.}
        \label{fig:ex1}
    \end{figure}
\end{example}

\subsection{Organization of the paper}

Section \ref{sec_problem} introduces the Christoffel-Darboux polynomial, its regularized version and our semi-algebraic approximant. In Section \ref{sec_main} the main results of the paper are collected, while their proofs are provided in Section \ref{sec_estimation}. More precisely, we first give some quantitative estimates on the support of the measure $\mu$ and then prove the $\Leb^1$ convergence of our approximant. Section \ref{sec_numeric} discusses computational issues and with the help of a simple Matlab prototype it illustrates the efficiency of our method on some examples. Finally, Section \ref{sec_conclusion} collects some concluding remarks together with further research lines to be followed.

\paragraph{Notation}

The Euclidean space of real-valued symmetric matrices of size $n$ is denoted by $\mathbb{S}^n$. Given a set $\X$ in Euclidean space, its diameter is denoted by $\mathrm{diam}(\X)$ and its volume, or Lebesgue measure, is denoted by $\mathrm{vol}(\X)$.
For $k\geq 1$, the Lebesgue space $\Leb^k(\X)$ consists of functions on $X$ whose $k$-norms are bounded.
Given a positive Borel measure $\mu$, we denote by $\spt{\mu}$ its support, defined as the smallest closed subset whose complement has measure zero. 

Throughout the paper, $p$ denotes the dimension of the ambient space.
Consistently with the notations introduced in Section \ref{sec:intro} for stating Problem \ref{eq:mainProblem}, we let $\genvar = (\xx \:\: y) \in \R^p$. We denote by $\R[\genvar]$ the algebra of multivariate polynomials of $\genvar \in \R^p$ with real coefficients. For a given degree $d\in\mathbb{N}$, the dimension of the vector space of polynomials of degree less than or equal to $d$ is denoted by $n_d:=\begin{pmatrix} p + d\\ d \end{pmatrix}$.

\section{Christoffel-Darboux approximation}
This section describes our main approximant based on the Christoffel-Darboux kernel. We first introduce notations and definitions, describe our regularization scheme for the Christoffel-Darboux kernel and then describe our functional approximant based on moments.

\label{sec_problem}
\subsection{Polynomials and moments}

Following the notations introduced for stating Problem \ref{eq:mainProblem}, any polynomial $q\in\R[\genvar]$ of degree $d$ can be expressed in the polynomial basis $\mathbf{b}(\genvar)$ as $q \colon \genvar \mapsto \mathbf{q}^\top \mathbf{b}(\genvar)$ with $\mathbf{q} \in\mathbb{R}^{n_d}$ denoting its vector of coefficients. Recall that the moment matrix of degree $2d$ of the measure $\mu$ reads
\[
\Mmat_{\mu,d}=\int \mathbf{b}(\genvar)\mathbf{b}(\genvar)^\top  d\mu(\genvar) \in \mathbb{S}^{n_d}.
\]
Since $\Mmat_{\mu,d}$ is positive semi-definite, it has a spectral decomposition 
\begin{equation}
				\label{eq:orthogonalDecomposition}
\Mmat_{\mu,d} = \mathbf{PEP}^\top,
\end{equation}
where $\mathbf{P} \in \R^{n_d\times n_d}$ is an orthonormal matrix whose columns are denoted $\mathbf{p}_i$, $i=1,2,\ldots,n_d$, and satisfy $\mathbf{p}_i^\top \mathbf{p}_i=1$ and $\mathbf{p}_i^\top \mathbf{p}_j=0$ if $i\neq j$, and $\mathbf{E} \in \mathbb{S}^{n(d)}$ is a diagonal matrix whose diagonal entries are eigenvalues $e_{i+1}\geq e_i \geq 0$ of the moment matrix. Each column $\mathbf{p}_i \in \R^{n_d}$ is the vector of coefficients of a polynomial $p_i\in\R[\genvar]$, $i=1,\ldots n_d$, so that
\begin{equation}\label{eq:ortho}
\begin{array}{llll}
				\mathbf{p}_i^\top \Mmat_{\mu,d}\,\mathbf{p}_i& = & e_i & = \int p_i^2(\genvar) d\mu(\genvar),\\
				\mathbf{p}_i^\top \Mmat_{\mu,d}\,\mathbf{p}_j&= & 0 & = \int p_i(\genvar) p_j(\genvar)d\mu(\genvar), \quad \,i \neq j.
\end{array}
\end{equation}

\subsection{Approximating the support from moments}

Let us assume for now that the support of the measure $\mu$ has nonempty interior, then $\Mmat_{\mu,d}$ is positive definite for any $d \in \NN$, i.e., $e_i > 0$, $i=1,\ldots, n_d$. In this case, one can define the Christoffel-Darboux polynomial
\begin{equation}
				q_{\mu,d}(\genvar):=\sum_{i=1}^{n_d} \frac{p^2_i(\genvar)}{e_i} = \mathbf{b}(\genvar)^\top \Mmat_{\mu,d}^{-1} \mathbf{b}(\genvar).
				\label{eq:CDpoly}
\end{equation}
It is known that sublevel sets of $q_{\mu,d}$ can be used to recover the support of $\mu$ for large $d$, see for example \cite{lasserre2017empirical} for an overview. 

The goal of this work is to approximate the function $f$. From a set theoretic perspective, this amounts to approximating the graph of $f$ whose closure is actually the support of the measure $\mu$ in \eqref{measure}. Hence the sublevel sets of $q_{\mu,d}$ are interesting candidates for this goal. However, in the case of the graph of a function, the construction given in \eqref{eq:CDpoly} is not valid anymore since this graph is a singular set so that $\Mmat_{\mu,d}$ may not be positive definite and invertible. In this singular setting, one should ideally consider the following extended value Christoffel-Darboux polynomial:
\begin{equation}
q^e_{\mu,d}(\genvar):=
\left\lbrace 
\begin{array}{ll}
+\infty &\text{ if } \exists i,\,e_i = 0,\, p_i(\genvar)  \neq 0\\
\sum_{e_i > 0} \frac{p^2_i(\genvar)}{e_i} = \mathbf{b}(\genvar)^\top \Mmat_{\mu,d}^\dagger \mathbf{b}(\genvar) &\text{ otherwise},
\end{array}
\right.	
\label{eq:CDpolySing}
\end{equation}
where $\dagger$ denotes the Moore-Penrose pseudo inverse. This is a natural extension of the Christoffel-Darboux polynomial to the singular case \cite{pauwels2018spectral} and amounts to working in the Zariski closure of the graph of $f$. 
\subsection{Regularization scheme}
\paragraph{Spectral filtering:}
Computing an object such as in \eqref{eq:CDpolySing} can be numerically sensitive since it essentially relies on pseudo-inverse which requires an eigenvalue thresholding scheme. This means that small perturbations of the moment matrix may lead to large changes of the output. Furthermore the candidate function takes finite values only on an algebraic set, and this situation is difficult to handle in finite precision arithmetic. One may rewrite the extended value polynomial \eqref{eq:CDpolySing} in the following form
\[
				q^e_{\mu,d}(\genvar) =\sum_{i=1}^{n_d} g(e_i) p^2_i(\genvar)
\]
where $ g \colon [0, +\infty) \mapsto [0, +\infty]$ with $g(s)= \frac{1}{s}$ for any $s >0$ and $g(0) = +\infty$. One approach to restore stability is to use regularization techniques which replace the pseudo-inversion expressed through the function $g$, by spectral filtering expressed through a different spectral function (see for example \cite{devito2014learning} for an illustration in the context of support estimation). Since the function $g$ is not regular, instead of studying the above defined extended value polynomial, one rather looks at the following polynomial
\begin{equation} \label{eq:CDpolyFilter}
				\sum_{i=1}^{n_d} g_\regparam(e_i) p^2_i(\genvar) 
\end{equation}
where $g_{\regparam}$ is a parametrized family of spectral filtering regular functions satisfying, for any $\regparam > 0$, $ g_\regparam \colon [0, +\infty) \mapsto [0, +\infty]$. Common examples include
\begin{align*}
				\text{Tikhonov regularization:} \qquad& g_\regparam \colon s \mapsto \frac{1}{\regparam + s},\\
				\text{Spectral cut-off:} \qquad& g_\regparam \colon s \mapsto \frac{1}{\regparam} \mathbb{I}_{[0,\regparam]}(s) + 				\frac{1}{s} \mathbb{I}_{(\regparam,+\infty)}(s),\\
				\text{Ideal low-pass:} \qquad& g_\regparam \colon s \mapsto \frac{1}{\regparam} \mathbb{I}_{[0,\regparam]}(s).
\end{align*}
We choose to work with the Tikhonov regularization as it has an intuitive measure space intepretation. We believe that our results can be generalized to different spectral filters.

\paragraph{Tikhonov regularization and measures:} Applying the Tikhonov spectral filter to \eqref{eq:CDpolyFilter} yields the following polynomial
\begin{align}
				\genvar \mapsto \sum_{i=1}^{n_d} \frac{p^2_i(\genvar)}{e_i+ \regparam}  = \mathbf{b}(\genvar)^\top (\Mmat_{\mu,d} + \regparam \Imat_{n_d})^{-1} \mathbf{b}(\genvar)
				\label{eq:regularizedCDpolyIntro}
\end{align}
where $\Imat_{n_d}$ denotes the identity matrix of size $n_d$.
In order to use analytic tools, we need to provide an interpretation of the addition of diagonal elements in terms of measures. One therefore has to choose a polynomial basis for which the diagonal matrix is the moment matrix of a reference Borel measure on $\R^p$ that we will denote $\mu_0$. We make the following assumption which will be standing throughout the paper.
\begin{assumption}[{\bf Reference measure and polynomial basis}]\hfill
\begin{itemize}
\item The reference measure $\mu_0$ is absolutely continuous with respect to the Lebesgue measure and it has compact support.
\item 
The polynomial basis $\mathbf{b}$ is orthonormal with respect to the bilinear form induced by $\mu_0$, that is  $\int b_i(\genvar)b_j(\genvar)d\mu_0(\genvar) = 1$ if $i=j$ and $0$ otherwise.
\end{itemize}
\label{ass:polynomialBasisMuZero}
\end{assumption}
The first part of Assumption \ref{ass:polynomialBasisMuZero} ensures that the moment matrix of $\mu_0$ is always positive definite. 
The second part of Assumption \ref{ass:polynomialBasisMuZero} provides the following relation:
\begin{align}
\Mmat_{\mu,d} + \regparam \Imat_{n_d} = \Mmat_{\mu + \regparam \mu_0,d}.
\label{eq:momentMatrixCorrespondance}
\end{align}

 An easy example of such a measure should be the following: considering a function $f$ whose domain of definition is contained in the unit cube of dimension $p-1$ and which takes values in $[-1,1]$, one might pick the uniform measure on the unit cube of dimension $p$, for which moments are easy to compute. 

Most importantly, using the notation in \eqref{eq:CDpoly}, this allows to express the polynomial of interest \eqref{eq:regularizedCDpolyIntro} as follows.
\begin{definition}
[{\bf Regularized Christoffel-Darboux polynomial}]
The regularized Christoffel-Darboux polynomial is the SOS
\begin{equation}
		\label{eq:regularizedCDpoly}
		q_{\mu+\regparam \mu_0,d}(\genvar) := \sum_{i=1}^{n_d} \frac{p^2_i(\genvar)}{e_i+\regparam }.
\end{equation}
\end{definition}
This provides a geometric interpretation of the regularization parameter as a combination of two measures: $\mu$ which is supported on the graph of the function of interest and $\mu_0$ which is a reference measure, used for regularization purposes.
The supported boundedness hypothesis in Assumption \ref{ass:polynomialBasisMuZero} will allow to provide quantitative estimates in further sections and it could in principle be replaced by a fast decreasing tail condition. An important example for measures satisfying Assumption \ref{ass:polynomialBasisMuZero} is the restriction of Lebesgue measure to the unit hypercube together with Legendre polynomials which form an orthonormal basis.

Making Assumption \ref{ass:polynomialBasisMuZero} is a slight restriction for which a few comments are in order. Firstly, this could be relaxed in various ways to remove the restriction on the polynomial basis, for example:
\begin{itemize}
				\item Replace the identity matrix by the moment matrix of $\mu_0$;
				\item Add assumption on the spectrum of the moment matrix $\mu_0$.
\end{itemize}
These would lead to a lot of technical complications and we find our results clearer and easier to state under Assumption \ref{ass:polynomialBasisMuZero}. Secondly, working numerically with polynomials in the standard monomial basis is problematic in many situations. Better conditioned polynomial bases are often those enjoying orthonormality properties with respect to a certain reference measure, such as e.g. Chebyshev or Legendre polynomials. We would like to argue here that the restrictions induced by  Assumption~\ref{ass:polynomialBasisMuZero} are quite benign since it is already common in practice to work in such polynomial bases for numerical reasons.

\subsection{Semi-algebraic approximant}

\begin{definition}[{\bf Semi-algebraic approximant}]\label{def:mainapproximant}
The regularized Christoffel-Darboux semi-algebraic approximant $f_{\regparam,d}$ is defined as follows:
\begin{equation}
\label{eq:mainapproximant}
				\mathbf{x} \in \X \mapsto  f_{\regparam,d}(\mathbf{x}) := \min \{ \displaystyle\mathrm{argmin}_{y\in\Y} \:\: q_{\mu + \regparam \mu_0,d}(\mathbf{x},y) \}.
\end{equation}
\end{definition}

\begin{remark}
				If $\X$ and $\Y$ are semi-algebraic, then the set-valued map which associates to each $\mathbf{x} \in \X$ the set
				\begin{equation*}
								\displaystyle\mathrm{argmin}_{y\in\Y} \:\: q_{\mu + \regparam \mu_0,d}(\mathbf{x},y)
				\end{equation*}
				is semi-algebraic. Recall that a map is semi-algebraic if its graph is semi-algebraic. By the Tarski--Seidenberg Theorem (see for example \cite[Theorem 2.6]{coste2002introduction}), any first order formula involving semi-algebraic sets describes a semi-algebraic set. Since minima are described by first order formulas, the argmin of a polynomial on the compact semi-algebraic set $\Y$ is a compact semi-algebraic subset of $\Y$, which is itself a subset of the real line. Hence the argmin set has a minimal element and the function in \eqref{eq:mainapproximant} is well defined.
				\label{rem:semi-algebraic}
\end{remark}

\begin{remark}
                For clarity of exposition we describe our main results by considering that the partial minimization in $y$ over $\Y$ is exact.
                As will be seen from arguments in the proof,  especially in the proof of Lemma \ref{lem:lemmaId}, approximate minimization up to a factor of the order $ d^{p+2}$ enjoys similar approximation properties.  Indeed, in Remark \ref{rem:robustPrecisionArgmin2}, we mention the precision $\frac{\gamma_d}{2}$, where $\gamma_d$ is chosen later on with more justifications in \eqref{eq:defGammad} to be of order $d^{p+2}$.  
                \label{rem:robustPrecisionArgmin}
\end{remark}

The two parameters $d$ and $\regparam$ control the behavior of the approximant $ f_{\regparam,d}$. In latter sections, we describe an explicit dependency between $d$ and $\regparam$ which allows to construct a sequence of regularization parameters $(\regparam_d)_{d\in \NN}$, and we investigate the asymptotic properties of the sequence of approximants $ \left( f_{\regparam_d,d} \right)_{d \in \NN}$.

\section{Main results}

\label{sec_main}
Our main theoretical contribution is an investigation of the relations between the function $f$ to be approximated and its regularized Christoffel-Darboux approximant $ f_{\regparam,d}$ under Assumption \ref{ass:polynomialBasisMuZero}. In particular we are interested in building an explicit sequence $\left( \regparam_d \right)_{d \in \RR}$ 
and investigating the convergence $f_{\beta_d,d}(\mathbf{x}) \to f(\mathbf{x})$  for 
$\mathbf{x} \in \X$, fixed, as well as the convergence $\|f -  f_{\beta_d,d} \|_{\Leb^1(\X)}\to 0$, when $d\to\infty$.

\subsection{Convergence under continuity assumptions} 
The following section describes our main result regarding convergence of the approximant $ f_{\regparam_d,d}$ in \eqref{eq:mainapproximant} under different regularity assumptions for the function $f$ to be approximated. Let us define
\[
\delta_0 := \mathrm{diam}(\spt{\mu + \mu_0}), \quad m:= \mu(\RR^p), \quad m_0:=\mu_0(\RR^p).
\]
\begin{theorem}
\label{thm-convergence} 
Under Assumption \ref{ass:polynomialBasisMuZero} and with the choice $\regparam_d = 2^{3 - \sqrt{d}}$ in Definition \ref{def:mainapproximant}, it holds:
\begin{itemize}
\item[(i)] If the set $S \subset \X$ of continuity points of $f$ is such that $\X \setminus S$ has Lebesgue measure zero, then
\[
 f_{\regparam_d,d}(\xx)\underset{d \to \infty}{\to} f(\xx)
\]
for almost all $\xx \in \X$, and
\[
\Vert f- f_{\regparam_d,d} \Vert_{\Leb^1(\X)} \underset{d \to \infty}{\to} 0.
\]
\item[(ii)] If $f$ is $L$-Lipschitz on $\X$ for some $L > 0$, then for any $d > 1$ and any $r > p$,
\begin{align*}
&\Vert f- f_{\regparam_d,d}\Vert_{ \Leb^1(\X)}\\
\leq\,&  \mathrm{vol}(\X)  \frac{\delta_0}{\sqrt{d} - 1}\left( 1 + L \right) + \mathrm{diam}(\Y) \frac{8(m+m_0)(3r)^{2r} e^{\frac{p^2}{d}}}{p^p e^{2r-p}d^{r-p} }.
\end{align*}
\end{itemize}
\end{theorem}
\begin{remark}\hfill
\begin{itemize}
\item As proved recently in \cite{silvestre2019oscillation}, the solutions to scalar conservation laws are continuous almost everywhere, i.e. the Lebesgue measure of the set of discontinuity points reduces to $0$. It is then clear that Item $(i)$ of Theorem \ref{thm-convergence} can be applied directly to the case of scalar conservation laws.
\item Thanks to Egorov's Theorem, pointwise convergence almost everywhere implies almost uniform convergence, that is uniform convergence up to a subset of $\X$ whose measure can be taken arbitrarily small. Since we manipulate bounded functions, this in turn implies convergence in $\Leb^1$.
\item For Lipschitz continuous functions, we obtain an $O(d^{-1/2})$ convergence rate in $\Leb^1$ norm by setting $r = p+1/2$.
\item The convergence rate for Lipschitz functions is slow and we observe in practice a much faster convergence. We conjecture that faster rates can be obtained for special classes of functions.
\end{itemize}
\label{rem:mainTheorem}
\end{remark}
Theorem \ref{thm-convergence} is a special case of a more general result described  in Theorem \ref{th:convergenceRobust}  and proven later on.

\subsection{Convergence for univariate functions of bounded variation}
Spaces of functions of bounded variations are of interest because many PDE problems are formulated on such spaces, see \cite{ambrosio2000functions} for an introduction. Modern construction of such spaces is done by duality through measure theoretic arguments. Our main proof mechanisms rely on pointwise properties of the function $f$, which are not completely captured by the measure theoretic construction. 

We prove $\Leb^1$ convergence for univariate bounded variation functions. The reason we are limited to the univariate setting is that we can use the classical definition of total variation which is directly connected to pointwise properties of the function of interest. We conjecture that the proposed approximation scheme is also convergent for multivariate functions of bounded variation, but we leave this question for future work.

\begin{definition}
Let $f \colon \RR \mapsto \RR$ be a measurable function. The (Jordan) total variation norm of $f$ is given by
\begin{align*}
V(f) = \sup_{n \in \NN}\sup_{x_0 <x_1<x_2<\ldots<x_n} \sum_{i=1}^n |f(x_i) - f(x_{i-1})|.
\end{align*}
\label{def:TV}
\end{definition}
\begin{theorem}
Under Assumption \ref{ass:polynomialBasisMuZero}, let $\X$ and $\Y$ be intervals of the real line, and assume that $V(f) < + \infty$.
With the choice $\regparam_d = 2^{3 - \sqrt{d}}$ in Definition \ref{def:mainapproximant}, for any $r >  2$ and for any $d > 1$, it holds
\begin{align*}
&\Vert f- f_{\regparam_d,d} \Vert_{\Leb^1(\X)}\\
\leq\,& \mathrm{vol}(\X) \left( \frac{2\delta_0}{\sqrt{d} - 1} + d^{-\frac{1}{4}} \right) \\
&+ \mathrm{diam}(\Y) \left( \frac{8(m+m_0) (3r)^{2r} e^{\frac{ 4}{d}}}{ 4 e^{2r- 2}d^{r- 2} } + \frac{4d^{\frac{1}{4}} V(f) \delta_0}{\sqrt{d} - 1 }\right).
\end{align*}
\label{th:BV1D} 
\end{theorem}
We remark that we obtain a convergence rate in $O(d^{-1/4})$. This result is a special case of a more general result described  in Theorem \ref{th:BV1DRobust}  and proven later on.

\subsection{Robustness to small perturbations}

In many situations, one only has access to an approximation of the regularized Christoffel-Darboux polynomial. This is for example the case when applying the Moment-SOS hierarchy. At the end, one indeed obtains pseudo-moments of the measure under consideration, i.e. a vector of real numbers close to the actual moments of the measure.
The moment matrix is then not $\Mmat_{\mu + \regparam_d \mu_0}$ as in \eqref{eq:momentMatrixCorrespondance} but a matrix $\Mmat$ close to it. The effect of this perturbation on the Christoffel-Darboux polynomial is captured by the following result.

\begin{lemma}
Assume that the approximate moment matrix $\Mmat \in \mathbb{S}^{n_d}$ is positive definite and let
\[
\alpha:=\|\Imat_{n_d} -  \Mmat_{\mu + \regparam_d \mu_0}^{\frac{1}{2}} \Mmat^{-1} \Mmat_{\mu + \regparam_d \mu_0}^{\frac{1}{2}}  \|
\]
where we used the matrix operator norm. Then the polynomial $q^\alpha_d \colon \genvar \mapsto \mathbf{b}(\genvar)^\top \Mmat^{-1} \mathbf{b}(\genvar)$ satisfy
\begin{align*}
\sup_{\genvar \in \RR^p} \left|1 - \frac{q^\alpha_d(\genvar)}{q_{\mu + \regparam_d \mu_0,d}(\genvar)}\right| \leq \alpha.\end{align*}
\label{lem:approxMomentMatrix}
\end{lemma}
\begin{proof}:
For any $\genvar \in \RR^p$, we have
\begin{align*}
&\left|\frac{q^\alpha_d(\genvar)}{q_{\mu + \regparam_d \mu_0 ,d}(\genvar)} -1\right| \\
=\quad & \frac{1}{q_{\mu + \regparam_d \mu_0 ,d}(\genvar)}\left| \mathbf{b}(\genvar)^\top (\Mmat^{-1} - \Mmat_{\mu + \regparam_d \mu_0}^{-1}) \mathbf{b}(\genvar)\right|\\
= \quad &\frac{1}{q_{\mu + \regparam_d \mu_0 ,d}(\genvar)}\left| (\Mmat_{\mu + \regparam_d \mu_0}^{-\frac{1}{2}} \mathbf{b}(\genvar))^\top (\Imat_{n_d} -  \Mmat_{\mu + \regparam_d \mu_0}^{\frac{1}{2}} \Mmat^{-1} \Mmat_{\mu + \regparam_d \mu_0}^{\frac{1}{2}} ) \Mmat_{\mu + \regparam_d \mu_0}^{-\frac{1}{2}}  \mathbf{b}(\genvar)\right|\\
\leq\quad & \left\| \Imat_{n_d} -  \Mmat_{\mu + \regparam_d \mu_0}^{\frac{1}{2}} \Mmat^{-1} \Mmat_{\mu + \regparam_d \mu_0}^{\frac{1}{2}}  \right\| = \alpha.
\end{align*}
\end{proof}

\begin{remark}[Accuracy of the moments]
In this paper, we assume the existence of a positive number $\alpha$ which makes the link between an approximated moment matrix and the real one, which exists in general, but estimates for $\alpha$ are not known in general. Such an analysis has been performed for the moments themselves in \cite{nie2007complexity}, but not for moment matrices. Indeed, the bounds linking moment matrices and their corresponding depend nonlinearly on too many variables to obtain easily bounds on the approximated moment matrix and the real one. This is a topic of further investigation.
\end{remark}

More generally, we can consider a robust Christoffel-Darboux function satisfying the following inequality.

\begin{assumption}
For a given $\alpha \in [0,1)$, let $(\regparam_d)_{d\in \NN}$ be a sequence of positive numbers and $(q^\alpha_d)_{d \in \NN}$ be a sequence of continuous functions over $\RR^p$ such that for any $d \in \NN$ and any $\genvar \in \RR^p$, we have
\begin{align*}
\left|1 - \frac{q^\alpha_d(\genvar)}{q_{\mu + \regparam_d \mu_0,d}(\genvar)}\right|\leq \alpha.
\end{align*}
\label{ass:perturbation}
\end{assumption}
Note that Assumption \ref{ass:perturbation} ensures that
\begin{align}
(1-\alpha) q_{\mu + \regparam_d \mu_0,d}(\genvar) \leq q^\alpha_d(\genvar) \leq (1+\alpha) q_{\mu + \regparam_d \mu_0,d}(\genvar).
\label{eq:perturbation}
\end{align}
Furthermore, one can always choose $q^\alpha_d(\genvar)= q_{\mu + \regparam_d \mu_0,d}(\genvar)$ and $\alpha = 0$ which corresponds to the nominal case. The robust approximant is then defined similarly as in Definition \ref{def:mainapproximant}.

\begin{definition}[{\bf Robust semi-algebraic approximant}]\label{def:mainapproximantRobust}
Given a degree $d \in \NN$, a regularizing parameter $\regparam > 0$, and a scalar $\alpha \in [0,1)$, our robust approximant $ f^\alpha_{\regparam_d,d} $
is defined as follows:
\begin{equation}
\label{eq:mainapproximantRobust}
\mathbf{x} \in \X \mapsto  f^\alpha_{\regparam_d,d}(\mathbf{x})  := \min \{\displaystyle\mathrm{argmin}_{y\in\Y} q^\alpha_d(\mathbf{x},y) \}.
\end{equation}
\end{definition}

The main technical result of this paper is the following from which Theorem \ref{thm-convergence} directly follows.
\begin{theorem}
Under Assumptions \ref{ass:polynomialBasisMuZero} and \ref{ass:perturbation}, and with the choice $\alpha \in [0,1)$ and $\regparam_d = 2^{3-\sqrt{d}}$ in Definition \ref{def:mainapproximantRobust}, it holds:
\begin{itemize}
\item[(i)] If the set $S \subset \X$ of continuity points of $f$ is such that $\X \setminus S$ has Lebesgue measure zero, then
\begin{align*}
& f^\alpha_{\regparam_d,d}(\mathbf{x}) \underset{d \to \infty}{\to} f(\mathbf{x})
\end{align*}
for almost all $\mathbf{x} \in \X$, and
\begin{align*}				
&\Vert f- f^\alpha_{\regparam_d,d}\Vert_{\Leb^1(\X)} \underset{d \to \infty}{\to} 0.
\end{align*}
\item[(ii)] If $f$ is $L$-Lipschitz on $\X$ for some $L > 0$, then for any $d > 1$ and any $r > p$,
\begin{align*}
&\Vert f- f^\alpha_{\regparam_d,d}\Vert_{ \Leb^1(\X)}\\
\leq\,&  \mathrm{vol}(\X)  \frac{\delta_0}{\sqrt{d} - 1}\left( 1 + L \right) + \mathrm{diam}(\Y)\frac{1+\alpha}{1 - \alpha} \frac{8(m + m_0)(3r)^{2r} e^{\frac{p^2}{d}}}{p^p e^{2r-p}d^{r-p}}.
\end{align*}
\end{itemize}
\label{th:convergenceRobust}
\end{theorem}
Furthermore, we have the following robust convergence result for univariate functions of bounded variation, from which Theorem \ref{th:BV1D} directly follows.
\begin{theorem}
Under Assumptions \ref{ass:polynomialBasisMuZero} and \ref{ass:perturbation},  let $\X$ and $\Y$ be intervals of the real line and assume that $V(f) < + \infty$. With the choice $\alpha \in [0,1)$ and $\regparam_d = 2^{3-\sqrt{d}}$ in Definition \ref{def:mainapproximantRobust}, for any $r >  2$ and for any $d > 1$, it holds
\begin{align*}
&\Vert f- f^\alpha_{\regparam_d,d}\Vert_{\Leb^1(\X)}\\
\leq\,& \mathrm{vol}(\X) \left( \frac{2\delta_0}{\sqrt{d} - 1} + d^{-\frac{1}{4}} \right) \\
&+ \mathrm{diam}(\Y) \left( \frac{1+\alpha}{1 - \alpha}\frac{8(m+m_0) (3r)^{2r} e^{\frac{ 4}{d}}}{ 4  e^{2r- 2 }d^{r- 2}} + \frac{4d^{\frac{1}{4}}V(f) \delta_0}{\sqrt{d} - 1 } \right).
\end{align*}
\label{th:BV1DRobust} 
\end{theorem}
The next section is dedicated to the proof of these theorems.

\section{Proofs}

\label{sec_estimation}

This section is divided into several subsections. Subsection \ref{sec:estimspt} gives some quantitative results on the estimation of the support of $\mu$ which does not depend on the nature of $\mu$ and could be of independent interest. More precisely, we show that the regularized Christoffel-Darboux polynomial takes large values outside the support of $\mu$ and smaller values inside. This is expressed by describing properties of certain sublevel sets of the polynomial being close to the support of $\mu$. In subsection \ref{sec:estimfun} we translate these geometric results in functional terms. In subsection \ref{sec:proofs}, we prove our main results: the argument of the minimum of the regularized Christoffel-Darboux polynomial is close to the graph of the function $f$.

\subsection{Estimation of the support}\label{sec:estimspt}

In this section we build a polynomial sublevel set that will be instrumental for our proofs of convergence in functional terms. Note that in practice this sublevel set is not computed: we just focus on the argmin of the regularized Christoffel-Darboux polynomial. The contents of this section may be considered of independent interest.

For any $d \in \N$ and $r \in \N$ such that $r > p$  and for any $\alpha \in [0,1)$ , define	\begin{equation}
\label{eq:defGammad}
\gamma_d :=\frac{1-\alpha}{8(m+m_0)} \frac{e^{2r}d^r}{(3r)^{2r}}
\end{equation}
and
\begin{equation}
S_d:=\lbrace \genvar\in\mathbb{R}^p :  q^\alpha_d(\genvar) < \gamma_d\rbrace.
\label{eq:defLevelSet}
\end{equation}
We aim at proving that the sublevel set $S_d$ is approaching the support of $\mu$ as $d$ goes to infinity, with a given convergence rate. It is precisely quantified with the following result which is illustrated in Figure \ref{fig:illustrSetRecovery}.
\begin{theorem}
For $d>1$ it holds
\begin{itemize}
\item[(i)] 
\begin{align*}
\mu(\{\genvar \in\mathbb{R}^p : \genvar \not\in S_d\})\leq \frac{1+\alpha}{1 - \alpha}\frac{8(m+m_0)(3r)^{2r} e^{\frac{p^2}{d}}}{p^p e^{2r - p} d^{r-p}}.
\end{align*}
\item[(ii)] For any $\genvar \in S_d$, 
\begin{align*}
\mathrm{dist}(\genvar,\spt{\mu})\leq \frac{\delta_0}{\sqrt{d}-1}.
\end{align*}
\end{itemize}
\label{th:supportEstimation}
\end{theorem}


\begin{figure}
    \centering
    \includegraphics[width=.6\textwidth]{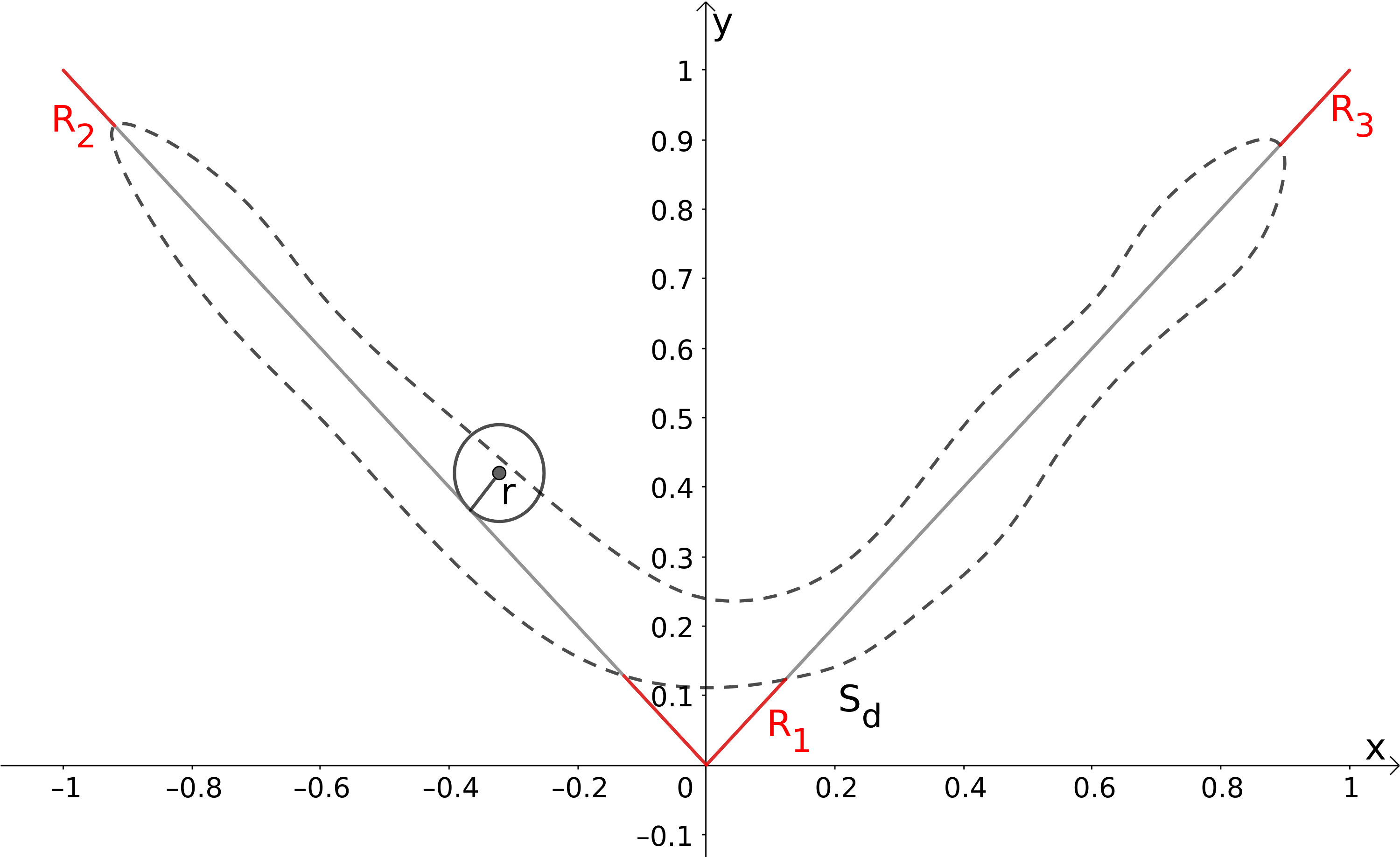}
    \caption{Illustration of the result of Theorem \ref{th:supportEstimation}, the dotted curve represents the boundary of the set $S_d$ and the considered function $f$ is the absolute value. The theorem states that (i) most of the points of the graph of $f$ will be in $S_d$ and that (ii) all points in $S_d$ will be close to the graph of $f$. More precisely, (i) states that the measure of $R_1\cup R_2\cup R_3$ will vanish and (ii) the distance $r$ of any point in $S_d$ to the graph of $f$ will go to zero for $d\to\infty$, respectively. }
    \label{fig:illustrSetRecovery}
\end{figure}

The results and techniques that we will use to prove this theorem are adapted from \cite{lasserre2017empirical} which considers the absolutely continuous setting without regularization.

\begin{proof}\textbf{ of Theorem \ref{th:supportEstimation} \textit{(i)}}
Using \eqref{eq:regularizedCDpoly} and \eqref{eq:ortho}, we obtain
\begin{align}
\int_{\mathbb{R}^p} q_{\mu + \regparam_d\mu_0,d}(\genvar)d\mu(\genvar) &= \sum_{i=1}^{n_d} \frac{e_i}{e_i+\regparam_d} \leq n_d \leq  d^p \left(\frac{e}{p}\right)^p e^{\frac{p^2}{d}}
\label{eq:propmarkov1}
\end{align}
where the last inequality is given in \cite[Lemma 6.5]{lasserre2017empirical}.
Using Markov's inequality \cite[Page 91]{stein2009real} and \eqref{eq:propmarkov1} yields
\begin{align}
\mu(\{\genvar\in\mathbb{R}^p : q_{\mu+\regparam_d \mu_0,d}(\genvar)\geq \frac{\gamma_d}{1+\alpha}\}) &\leq  \frac{\int_{\mathbb{R}^p} (1+\alpha)q_{\mu + \regparam_d\mu_0,d}(\genvar)d\mu(\genvar)}{\gamma_d}\nonumber\\
&\leq  (1 + \alpha) \frac{d^p \left(\frac{e}{p}\right)^p e^{\frac{p^2}{d}}}{\gamma_d}.
\label{eq:propmarkov2}
\end{align}
Now using \eqref{eq:perturbation} and \eqref{eq:defLevelSet}, we have the following implications 
\begin{align*}
\genvar \not \in S_d \quad\Leftrightarrow\quad q^\alpha_d(\genvar) \geq \gamma_d\quad \Rightarrow\quad q_{\mu+\regparam_d \mu_0,d}(\genvar) \geq \frac{\gamma_d}{1 + \alpha}.
\end{align*}
Hence $ \mu(\{\genvar \in\mathbb{R}^p : \genvar \not\in S_d\}) \leq \mu(\{\genvar\in\mathbb{R}^p : q_{\mu+\regparam_d \mu_0,d}(\genvar)\geq \frac{\gamma_d}{1+\alpha}\})$. Using the expression of $\gamma_d$ in \eqref{eq:defGammad} and the inequality \eqref{eq:propmarkov2}, one has:
$$
\mu(\{\genvar \in\mathbb{R}^p : \genvar \not\in S_d\})\leq \frac{1+\alpha}{1 - \alpha}\frac{8(m+m_0) (3r)^{2r} e^{\frac{p^2}{d}}}{p^p e^{2r - p}d^{r-p}}$$
which concludes the proof of item $(i)$ of Theorem \ref{th:supportEstimation}. 
\end{proof}
To carry out the proof of Theorem \ref{th:supportEstimation} \textit{(ii)}, we begin with a few lemmas. The following result is classical, see \textit{e.g.}  \cite[Remark 3.6.]{lasserre2017empirical} and \cite[Equation (1.1.)]{kroo2012christoffel}.
\begin{lemma}
Let $d \in \NN$, $\genvar \in \RR^p$, $\regparam > 0$, and $q$ be a polynomial of degree at most $d$. Then 
\begin{align*}
\frac{q^2(\genvar)}{ \int_{\RR^p} q^2(\mathbf{z}) d(\mu + \regparam\mu_0) (\mathbf{z})} \leq q_{\mu + \regparam \mu_0,d} (\genvar).
\end{align*}
\label{lem:lowerBoundCDPoly}
\end{lemma}
The following Lemma defines the needle polynomial, introduced first in \cite{kroo2012christoffel}, and gives a quantitative result crucial for our analysis.  
\begin{lemma}[Existence of a needle polynomial]
\label{lemma-existence-needle}
Let $B_\delta$ denote the euclidean ball of radius $\delta$. Then, for all $\delta\in (0,1)$ and $d\in\mathbb{N}$, $d>0$, there exists a polynomial $q$ of degree $2d$ such that
\begin{equation}
\begin{split}
&q(0)=1,\: q(\genvar) \in [-1,1]\text{ for all }\genvar\in B_1,\text{ and }\\
&|q(\genvar)|\leq 2^{1-\delta d}\text{ for all }\genvar\in B_1\setminus B_\delta.
\end{split} 
\end{equation}
\end{lemma}
A detailed proof is provided in \cite[Lemma 6.3]{lasserre2017empirical}. Thanks to the latter lemma, we can characterize the behavior of the regularized Christoffel-Darboux polynomial $q_{\mu + \regparam_d \mu_0,d}$ outside the support of $\mu$:
\begin{lemma}
\label{lem:lowerBoundOutsideSupport}
Let $d \in \NN $, $d > 1$ and $\genvar \in \RR^p$.  Recall that $\beta_d:=2^{3-\sqrt{d}}$.  Assume that $\mathrm{dist}(\genvar,\spt{\mu}) \geq \frac{\delta_0}{\sqrt{d}-1}$. Then
\begin{align}
\frac{ 2^{\sqrt{d}-3}}{m+m_0}\leq q_{\mu + \regparam_d \mu_0,d}(\genvar).
\label{eq:ineq-needle}
\end{align}
\end{lemma}

\begin{proof}\textbf{ of Lemma \ref{lem:lowerBoundOutsideSupport}:}
Let $d > 1$, $\genvar \in \RR^p$, $\delta = \mathrm{dist}(\genvar,\spt{\mu})$. Let $d' \in \NN$ and $t >0$, arbitrary for the moment. Consider the affine map $T:\mathbf{w}\mapsto \frac{\mathbf{w}-\genvar}{\delta + \delta_0}$. Let $q$ be the degree $2d'$ polynomial given as in Lemma \ref{lemma-existence-needle} such that
\begin{align}
 q(0)=1,\: q(\mathbf{w}) \in [-1,1]\text{ for all }\mathbf{w}\in B_1,\text{ and } \nonumber\\
 |q(\mathbf{w})|\leq 2^{1-\frac{\delta d'}{\delta + \delta_0}}\text{ for all }\mathbf{w}\in B_1\setminus B_{\frac{\delta }{\delta+ \delta_0}}.
\label{eq:lemma-needle1}
\end{align}
Let $r = q \circ T$. The polynomial $r$ satisfies
\begin{align}
 |r(\mathbf{z}')|  &\leq 1,\qquad \forall  \mathbf{z}' \in \spt{\mu + \mu_0},\nonumber\\
 r(\mathbf{z}') &\leq 2^{1-\frac{\delta d'}{\delta + \delta_0}},\qquad \forall  \mathbf{z}' \in \spt{\mu},\nonumber\\
r(\genvar) &= 1.
\label{eq:lemma-needle2}
\end{align}
Using Lemma \ref{lem:lowerBoundCDPoly} and the fact that $r(\mathbf{z})=1$ we obtain
\begin{align}
\left(  \int_{\RR^p} r^2(\mathbf{w}) d(\mu + \regparam\mu_0) (\mathbf{w}) \right)^{-1} \leq q_{\mu + \regparam\mu_0,2d'}(\genvar)  
\label{eq:lemma-needle3}
\end{align}

and
\begin{align}
\left(  \int_{\RR^p} r^2(\mathbf{w}) d(\mu + \regparam\mu_0) (\mathbf{w}) \right)^{-1} \leq q_{\mu + \regparam\mu_0,2d'+1}(\genvar)  
\label{eq:lemma-needle31}
\end{align}

From \eqref{eq:lemma-needle2}, we deduce
\begin{align}
\int_{\RR^p} r^2(\mathbf{w}) d(\mu+\regparam\mu_0) (\mathbf{w}) &\leq 2^{2-\frac{\delta(2d')}{\delta + \delta_0}} m + \regparam m_0
\label{eq:lemma-needle4}
\end{align} 
 
\begin{align}
\int_{\RR^p} r^2(\mathbf{w}) d(\mu+\regparam\mu_0) (\mathbf{w})  \leq 2^{3-\frac{\delta(2d'+1)}{\delta + \delta_0}}m + \regparam m_0
\label{eq:lemma-needle41}
\end{align}

Combining \eqref{eq:lemma-needle3}, \eqref{eq:lemma-needle31} , \eqref{eq:lemma-needle4}  and \eqref{eq:lemma-needle41}, we obtain the following bounds
\begin{align}
q_{\mu + \regparam \mu_0,2d'}(\genvar)&\geq \left(2^{3-\frac{\delta(2d')}{\delta + \delta_0}} m + \regparam m_0  \right)^{-1}, \nonumber \\
q_{\mu + \regparam \mu_0,2d'+1}(\genvar)&\geq \left(2^{3-\frac{\delta(2d'+1)}{\delta + \delta_0}} m + \regparam m_0  \right)^{-1}.
\label{eq:lemma-needle5}
\end{align}
Recall that $d'$ and $ \beta$ were arbitrary. Now we can choose $d' = \lfloor d/2\rfloor$, $ \beta = \regparam_d$ in one of the identities in \eqref{eq:lemma-needle5} (depending on the parity of $d$) to obtain
\begin{align}
q_{\mu + \regparam_d \mu_0,d}(\genvar)&\geq \left(2^{3-\frac{\delta d}{\delta + \delta_0}} m + \regparam_d m_0  \right)^{-1} \geq \frac{2^{\sqrt{d}-3}}{m+m_0},
\label{eq:lemma-needle6}
\end{align}
where the last inequality follows because the right hand side is strictly increasing as a function of $\delta$ and $\delta \geq \frac{\delta_0}{\sqrt{d} - 1}$. This proves the desired result.
\end{proof}
Let us give two additional simple technical lemmas.
\begin{lemma}
\label{lem:simple}
				For any $r > 0$,
				\begin{align*}
								\min_{x > 0} \:\{\log(2) x - (2r) \log(x)\} = \quad&(2r)\left( 1 - \log\left( \frac{2r}{\log(2)} \right) \right) \\
								\geq\quad &  (2r)\left( 1 - \log(3r) \right). 
				\end{align*}
				\label{lem:techLem1}
\end{lemma}
\begin{proof}\textbf{ of Lemma \ref{lem:simple}:}
				A simple analysis shows that the minimum is attained at $x = \frac{2r}{\log(2)}$. The lower bound follows because $\frac{2}{\log(2)} \leq 3$. 
\end{proof}
\begin{lemma}
\label{lem:simple2}
For any $d \in \NN$, we have
\begin{align*}
\frac{ 2^{\sqrt{d}-3}}{m+m_0} \geq \frac{\gamma_d}{1 - \alpha}.	
\end{align*}
\label{lem:techLem2}
\end{lemma}
\begin{proof}\textbf{ of Lemma \ref{lem:simple2}:}
Using the definition of $\gamma_d$ in \eqref{eq:defGammad}, we have
\begin{align}
\log\left(\gamma_d\frac{8(m+m_0)}{1-\alpha}  \right) &=  2r ( 1 - \log(3r) +  \log(\sqrt{d})) \nonumber \\
&\leq \log(2) \sqrt{d},
\label{eq:techLem21}
\end{align}
where the inequality follows from Lemma \ref{lem:techLem1} with $x = \sqrt{d}$. This proves the desired result.
\end{proof}

\begin{proof}\textbf{ of Theorem \ref{th:supportEstimation} \textit{(ii)}:}
We prove the result by contraposition. We have the following chain of implications for $\genvar \in \RR^p$,
\begin{align}
&\mathrm{dist}(\genvar,\spt{\mu})\geq \frac{\delta_0}{\sqrt{d}-1} \nonumber\\
\Rightarrow \quad & \frac{ 2^{\sqrt{d}-3}}{m+m_0}\leq q_{\mu + \regparam_d \mu_0,d}(\genvar)\nonumber \\
\Rightarrow \quad & \gamma_d/ (1 - \alpha) \leq q_{\mu + \regparam_d \mu_0,d}(\genvar)\nonumber \\
\Rightarrow \quad & \gamma_d \leq q^\alpha_d(\genvar)\nonumber\\
\Rightarrow \quad & \genvar \not \in S_d,
\label{eq:propSuffMoment}
\end{align}
where the first implication is from Lemma \ref{lem:lowerBoundOutsideSupport}, the second implication is due to Lemma \ref{lem:techLem2}, the third implication is deduced from \eqref{eq:perturbation} and the last implication is from the definition of $S_d$ in \eqref{eq:defLevelSet}.
\end{proof}

\subsection{Estimation of functions}\label{sec:estimfun}

We now translate Theorem \ref{th:supportEstimation} in functional terms. Considering that $\mathbf{z}$ can be written as follows $\mathbf{z}=(\mathbf{x},y)$, let us introduce a specific set which will be of interest throughout the proof:
\begin{equation}
\label{eq:defId}
I_d:=\{ \xx\in \X : \inf_{y\in\Y}\,q^\alpha_d(\xx,y)\geq \gamma_d\}.
\end{equation}
\begin{lemma}
 Suppose that $d>1$.  Then, we have
\begin{equation}
\int_{I_d}d\mathbf{x} \leq 
\frac{1+\alpha}{1 - \alpha}\frac{8(m+m_0) (3r)^{2r} e^{\frac{p^2}{d}}}{p^p e^{2r-p} d^{r-p}}.
\end{equation}
\label{lem:lemmaId}
\end{lemma}
\begin{proof}\textbf{ of Lemma \ref{lem:lemmaId}:}
For all $\xx\in\X$ and all $A\subset \mathbb{R}^{p}$ measurable, one has
\begin{equation}
\int_{\Y}\mathbb{I}_A(\xx,y)\delta_{f(\xx)}(dy)=\mathbb{I}_A(\xx,f(\xx))
\end{equation}
and hence
\begin{equation}
\begin{split}
\mu(A) = \int_{\Y} \mathbb{I}_A(\xx,y) d\mu(\xx,y) = &\int_{\X} \mathbb{I}_A(\xx,f(\xx))d\xx =  \int_{I_A} d\xx,
\end{split}
\end{equation}
where $I_A:=\lbrace \xx\in \X : (\xx,f(\xx))\in A\rbrace$. One can see from \eqref{eq:defId} that $\xx \in I_d$ implies that $(\xx,f(\xx)) \not \in S_d$ and hence $I_d \subset I_{S_d^c}$ where $S_d^c$ denotes the complement of $S_d$ given in \eqref{eq:defLevelSet}. We deduce that
\begin{equation}
\int_{I_d} d\xx \leq \int_{I_{S_d^c}} d\xx = \mu(S_d^c)
\end{equation}
and the result follows from Item (i) of Theorem \ref{th:supportEstimation}.
\end{proof}
\begin{remark}
Letting $\tilde{I}_d := \{ \xx \in \X : \inf_{y\in\Y}\,q^\alpha_d(\xx,y)\geq \gamma_d / 2\}$ it can be seen using the exact same arguments that a bound on $\int_{\tilde{I}_d}d\xx$ holds similarly as in Lemma \ref{lem:lemmaId} with a multiplicative factor of $2$. This can be used to handle the situation where the argmin in \eqref{eq:mainapproximant} is computed up to a precision of the order $\gamma_d/2$. See also Remark \ref{rem:robustPrecisionArgmin}.
            \label{rem:robustPrecisionArgmin2}
\end{remark}
Thanks to Lemma \ref{lem:lemmaId}, it is sufficient to prove the convergence of the approximated function in the set $I_d^c:=\X\setminus I_d$.
\begin{proposition}
\label{prop-almost}
Let $d\in\N$, $d > 1$. Let $\alpha \in [0,1)$ be as in Assumption \ref{ass:perturbation} with $\regparam_d = 2^{3-\sqrt{d}}$.
Consider $ f^\alpha_{\regparam_d,d}$ as in \eqref{eq:mainapproximantRobust} and let $I_d$ be defined by \eqref{eq:defId}. Then for any $r > p$, we have
\begin{align*}
\int_{\X} |f(\xx)- f^\alpha_{\regparam_d,d}(\xx)|d\xx \leq  &\int_{I_d^c} |f(\xx)- f^\alpha_{\regparam_d,d}(\xx)| d\xx \\
&+  \mathrm{diam}(\Y)\frac{1+\alpha}{1 - \alpha} \frac{8(m+m_0) (3r)^{2r} e^{\frac{p^2}{d}}}{p^p e^{2r-p} d^{r-p}}.
\end{align*}
\end{proposition}

\begin{proof}\textbf{ of Proposition \ref{prop-almost}:}
				Since $y$ takes values in the compact set $\Y$, it is clear that $ f_{\regparam_d,d}^\alpha \in \Leb^\infty(\X)$. We have 
\begin{align}
&\int_{\X} |f(\mathbf{x})- f_{\regparam_d,d}^\alpha(\mathbf{x})|d\mathbf{x}  \nonumber\\
= &\int_{\X\setminus I_d} |f(\mathbf{x})- f_{\regparam_d,d}^\alpha(\mathbf{x})| d\mathbf{x} + \int_{I_d} |f(\mathbf{x})- f_{\regparam_d,d}^\alpha(\mathbf{x})|d\mathbf{x}\nonumber\\
\leq & \int_{I_d^c} |f(\mathbf{x})- f_{\regparam_d,d}^\alpha(\mathbf{x})| d\mathbf{x} + \Vert f- f_{\regparam_d,d}^\alpha\Vert_{\Leb^\infty(I_d)}\int_{I_d} d\mathbf{x}\nonumber\\
\leq &  \int_{I_d^c} |f(\mathbf{x})- f_{\regparam_d,d}^\alpha(\mathbf{x})| d\mathbf{x} \nonumber\\
&+  \mathrm{diam}(\Y)\frac{1+\alpha}{1 - \alpha} \frac{8(m+m_0) (3r)^{2r} e^{\frac{p^2}{d}}}{p^p e^{2r-p} d^{r-p}}   \label{convergence-ae}
\end{align}
where we have used Lemma \ref{lem:lemmaId} for the last inequality. 
\end{proof}

\subsection{Proofs of the main theorems}\label{sec:proofs}

We are now in position of proving Theorem \ref{th:convergenceRobust}. We start with the Lipschitz case, which is the simplest and conveys most of the ideas.
\begin{proof}\textbf{ of Theorem \ref{th:convergenceRobust} (ii) Using Proposition \ref{prop-almost}}: It remains to bound the term
\begin{equation}
\int_{I_d^c} |f(\mathbf{x})- f_{\regparam_d,d}^\alpha(\mathbf{x})| d\mathbf{x} .
\end{equation}
For any $\mathbf{x} \in \X$ define
\begin{equation}
\mathbf{u}_d(\mathbf{x}) \in \mathrm{argmin}_{\mathbf{u}\in \X} \left\|\left(\mathbf{x},\:  f_{\regparam_d,d}^\alpha(\mathbf{x})\right)-\left(\mathbf{u},\: f(\mathbf{u})\right)\right\|,
\end{equation}
with an arbitrary choice in the case where the argmin is not unique. Note that by continuity, the graph of $f$ is closed so that the minimum is attained. Using the definition of $I_d$ in \eqref{eq:defId}, the fact that $\mathbf{x} \in I_d^c$ implies that 
\begin{equation}
\left(\mathbf{x},\:  f_{\regparam_d,d}^\alpha(\mathbf{x}) \right) \in S_d.
\end{equation}
Moreover, Theorem \ref{th:supportEstimation} implies that

\begin{align}
\label{eq:mainIneqSupportSuccess}
| f_{\regparam_d,d}^\alpha(\mathbf{x})-f(\mathbf{u}_d(\mathbf{x}))|&\leq \frac{\delta_0}{\sqrt{d} - 1},\\
\Vert \mathbf{x}-\mathbf{u}_d(\mathbf{x})\Vert &\leq \frac{\delta_0}{\sqrt{d} - 1}, \nonumber
\end{align}

 where $\Vert \cdot\Vert$ denotes the usual Euclidean distance in $\mathbb{R}^{p-1}$. 
Therefore, using Lipschitz continuity of $f$, we have, for any $\mathbf{x} \in I_d^c$,
\begin{align}
| f_{\regparam_d,d}^\alpha(\mathbf{x}) - f(\mathbf{x})|&\leq | f_{\regparam_d,d}^\alpha(\mathbf{x})-f(\mathbf{u}_d(\mathbf{x}))| + |f(\mathbf{x})-f(u_d(\mathbf{x}))|\nonumber\\
&\leq   \frac{\delta_0}{\sqrt{d} - 1}\left( 1 + L \right).
\label{first-inequality}
\end{align}
We deduce that
\begin{align}
\int_{I_d^c} |f(\mathbf{x})- f_{\regparam_d,d}^\alpha(\mathbf{x})| d\mathbf{x}  \leq  \mathrm{vol}(\X)  \frac{\delta_0}{\sqrt{d} - 1}\left( 1 + L \right)				\label{eq:boundIdcIntegral}
\end{align}
which concludes the proof.
\end{proof}

We now turn to case (i), starting with the pointwise convergence.

\begin{proof}\textbf{ of Theorem \ref{th:convergenceRobust} (i):} 
We rely on a slighlty different use of Lemma \ref{lem:lemmaId}.
Choose $r = p + 2$ and let
\begin{align}
I := \left\{ \mathbf{x} \in \X : \forall d_0 \in \NN, \,\exists d \in \NN,\, d \geq d_0, \, \mathbf{x} \in I_d \right\} = \cap_{d_0 \in \NN} \cup_{d \geq d_0} I_d.
\label{eq:defI}
\end{align}
Lemma \ref{lem:lemmaId} ensures that $\mathrm{vol}(I_d) = O(1/d^2)$ so that 
\begin{align*}
\mathrm{vol}\left( \cup_{d \geq d_0} I_d \right) \leq \sum_{d \geq d_0} \mathrm{vol}(I_d) \quad \underset{d_0 \to \infty}{\to} 0.
\end{align*} 
We have $\mathrm{vol}(I) = \lim_{d_0 \to \infty}\left(  \cup_{d \geq d_0} I_d  \right) = 0$. This means that we have the two following properties, for almost every $\mathbf{x} \in \X$:
\begin{itemize}
\item $f$ is continuous at $\mathbf{x}$ (by assumption),
\item $\exists d_0 \in \NN$, $\forall d \in \NN$, $d \geq d_0$, $\mathbf{x} \not \in I_d$ (because $\mathrm{vol}(I) = 0)$.
\end{itemize}
Fix any such $\mathbf{x}$ and for any $d \geq d_0$ consider
\begin{equation}
(\mathbf{u}_d,v_d) \in \mathrm{argmin}_{\mathbf{u}\in \X, v \in \Y} \left\|\left(\mathbf{x},\:  f_{\regparam_d,d}^\alpha(\mathbf{x}) \right)-\left( \mathbf{u},\: v \right)\right\|, \qquad \mathrm{s.t.} \qquad (\mathbf{u},v) \in \spt{\mu}
\end{equation}
with an arbitrary choice when the argmin is not unique.
Note that the support of $\mu$ is actually the closure of the graph of $f$ so that the minimum is attained. Using the definition of $I_d$ in \eqref{eq:defId}, we have that $\mathbf{x} \in I_d^c$ for all $d \geq d_0$ which implies that 
\begin{equation}
\left(\mathbf{x},\:  f_{\regparam_d,d}^\alpha(\mathbf{x}) \right) \in S_d,
\end{equation}
and Theorem \ref{th:supportEstimation} implies that
\begin{align*}
| f_{\regparam_d,d}^\alpha(\mathbf{x})-v_d| &\leq \frac{\delta_0}{\sqrt{d} - 1},\\
 \Vert \mathbf{x}-\mathbf{u}_d\Vert &\leq \frac{\delta_0}{\sqrt{d} - 1}. \nonumber
\end{align*}
Since $(\mathbf{u}_d,v_d) \in \spt{\mu}$ and $\spt{\mu}$ is the closure of the graph of $f$, there exists $\mathbf{h}_d \in \X$, such that
\begin{align}
| f_{\regparam_d,d}^\alpha(\mathbf{x})-f(\mathbf{h}_d)| &\leq \frac{2\delta_0}{\sqrt{d} - 1},\nonumber\\
 \Vert \mathbf{x}-\mathbf{h}_d\Vert  &\leq \frac{2\delta_0}{\sqrt{d} - 1}.
\label{eq:mainIneqNonContinuous}
\end{align}
This concludes the proof of pointwise convergence since
\begin{align*}
| f_{\regparam_d,d}^\alpha(\mathbf{x}) - f(\mathbf{x})|&\leq |f_{\regparam_d,d}^\alpha(\mathbf{x})-f(\mathbf{h}_d)| + |f(\mathbf{x})-f(\mathbf{h}_d)|\nonumber
\end{align*}
and both terms tend to $0$ as $d \to \infty$, using \eqref{eq:mainIneqNonContinuous} and continuity of $f$ at $\mathbf{x}$.

Convergence in $\Leb^1$ follows from Egorov's Theorem (see \textit{e.g.} \cite[chapter 18]{royden2010real}): for any $\epsilon > 0$, there exists $S_\epsilon \subset \X$, measurable, of Lebesgue measure smaller than $\epsilon$ such that $ f_{\regparam_d,d}^\alpha \to f$ uniformly on $\X \setminus S_\epsilon$. We have
\begin{align*}
\Vert  f_{\regparam_d,d}^\alpha - f \Vert_{\Leb^1(\X)} &= \int_{\X} | f_{\regparam_d,d}^\alpha(\mathbf{x}) - f(\mathbf{x})| d\mathbf{x} \\
&= \int_{S_{\epsilon}} |f_{\regparam_d,d}^\alpha(\mathbf{x}) - f(\mathbf{x}) d\mathbf{x}  + \int_{\X \setminus S_{\epsilon}} |f_{\regparam_d,d}^\alpha(\mathbf{x}) - f(\mathbf{x})| d\mathbf{x} \\
&\leq \mathrm{vol}(S_\epsilon)\: \mathrm{diam}(\Y) + \mathrm{vol}(\X) \Vert f_{\regparam_d,d}^\alpha - f(\mathbf{x}) \Vert_{\Leb^\infty(\X \setminus S_{\epsilon})} \\
&\leq \epsilon\  \mathrm{diam}(\Y) + \mathrm{vol}(\X) \Vert f_{\regparam_d,d}^\alpha - f(\mathbf{x}) \Vert_{\Leb^\infty(\X \setminus S_{\epsilon})}. 
\end{align*}
By uniform convergence the second term goes to $0$ as $d \to \infty$, this shows that $$\lim\sup_{d \to \infty}\Vert f_{\regparam_d,d}^\alpha- f \Vert_{\Leb^1(\X)} \leq \epsilon\ \mathrm{diam}(\Y).$$ Moreover, since $\epsilon > 0$ was arbitrary, the limit is $0$.
\end{proof}

We now turn to the proof of convergence in $\Leb^1$ norm for univariate functions of bounded variation.
\begin{lemma}
\label{lem:packing}
Let $f \colon \RR \mapsto \RR$ be such that $V(f)$ is finite and $f$ vanishes outside a segment $I$. Let $a,b >0$ be positive constants. Let
\begin{align*}
J = \left\{ t \in I : \exists u, |t-u| \leq b, u \in I, |f(u) - f(t)|>a \right\}.
\end{align*}
Then
\[
\int_{J}dx \leq 2V(f)\frac{b}{a}.
\]
\label{lem:techLemBVfun}
\end{lemma}

\begin{proof}\textbf{ of Lemma \ref{lem:packing}:}
This is a packing argument.
Let $J_0 = J$ and follow the recursive process, for $k \in \NN$, $k \geq 1$,
\begin{align*}
\begin{cases}
\text{if }&J_{k-1} \cap J \neq \emptyset\\
&\text{let } t_k \in J_{k-1} \cap J,\, u_k \in J,\, |t_k-u_k| \leq b\\
&\text{let } J_k = J_{k-1} \setminus [t_k-b,t_k+b]\\
\text{otherwise }&\text{stop}.
\end{cases}
\end{align*}
This process must stop after a finite number of iterations. Indeed, the set $I_k = I \setminus \cup_{i=1}^k [t_k-b,t_k+b]$ consists of a finite union of intervals. At iteration $k$, either one of these intervals is a subset of $[t_{k+1}-b,t_{k+1}+b]$ and then it is removed entirely from $I_k$, or otherwise $t_{k+1}$ is contained in an interval which contains either $[t_{k+1}-b,t_{k+1}]$ or  $[t_{k+1},t_{k+1}+b]$ and the Lebesgue measure of $I_{k+1}$ is reduced by at least $b$ compared to $I_k$, possibly creating a new interval.

Let $K$ be the last iteration, so that $J_{K} \cap J = \emptyset$. By the iterative process, at each step, the measure of $J_k$ is reduced by at most $2b$ and we have
\begin{align*}
0 = \int_{J_K\cap J} dx \geq \int_{J_{K-1} \cap J} dx - 2 b \geq \ldots \geq \int_{J_{0} \cap J} dx - 2Kb
\end{align*}
so that
\begin{align*}
\int_{J}dx \leq 2Kb.
\end{align*}
Finally, since the intervals $[t_k,u_k]$, $k = 1,\ldots, K$ are disjoint, we have
\begin{align*}
V(f) \geq \sum_{i=1}^K |f(t_k) - f(u_k)| \geq K a \geq \frac{a}{2b} \int_{J}dx,
\end{align*}
which proves the desired result.
\end{proof}

\begin{proof}\textbf{ of Theorem \ref{th:BV1DRobust}:}
For any $d > 1$ consider \begin{align*}
J_d = \left\{ x \in \X : \exists u \in \X, |x-u| \leq  2\delta_0 / (\sqrt{d} - 1),  |f(u) - f(t)|>d^{-1/4} \right\}.
\end{align*}
By Lemma \ref{lem:techLemBVfun}, for any $d > 1$ we have 
\begin{align}
 \mathrm{vol}(J_d) \leq \frac{4\delta_0 d^{1/4}V(f)}{\sqrt{d} - 1}.
\label{eq:Jd}
\end{align}
Choose any $d >1$ and any $x$ such that $x \not \in I_d$ and $x \not \in J_d$.  Consider 
\begin{equation}
(u_d,v_d) \in \mathrm{argmin}_{u\in \X, v \in \Y} \left\|\left( x,\: f_{\regparam_d,d}^\alpha(x)\right)-\left( u,\: v \right) \right\| \quad \mathrm{s.t.} \quad (u,\:v) \in \spt{\mu} 
\end{equation}
with an arbitrary choice when the argmin is not unique.

Note that the support of $\mu$ is actually the closure of the graph of $f$ so that the minimum is attained. Using the definition of $I_d$ in \eqref{eq:defId}, $x \in I_d^c$ implies that 
\begin{equation}
\left( x,\: f_{\regparam_d,d}^\alpha(x) \right) \in S_d
\end{equation}
and Theorem \ref{th:supportEstimation} implies that
\begin{align*}
|f_{\regparam_d,d}^\alpha(x)-v_d| &\leq \frac{\delta_0}{\sqrt{d} - 1},\\
|x-u_d| &\leq \frac{\delta_0}{\sqrt{d} - 1}.
\end{align*}
Since $(u_d,v_d) \in \spt{\mu}$ and $\spt{\mu}$ is the closure of the graph of $f$, there exists $ h_d \in \X$, such that
\begin{align}
|f_{\regparam_d,d}^\alpha(x)-f(h_d)| &\leq \frac{2\delta_0}{\sqrt{d} - 1},\nonumber\\
|x-h_d| &\leq \frac{2\delta_0}{\sqrt{d} - 1}.
\label{eq:mainIneqBV}
\end{align}
Now since $x \not \in J_d$ and $|x-h_d| \leq \frac{2\delta_0}{\sqrt{d} - 1}$, we have $|f(x) - f(h_d)|\leq d^{-1/4}$. This entails
\begin{align*}
|f_{\regparam_d,d}^\alpha(x) - f(x)|&\leq |f_{\regparam_d,d}^\alpha(x)-f(h_d)| + |f(x)-f(h_d)|\\
&\leq \frac{2\delta_0}{\sqrt{d} - 1} + d^{-1/4}.
\end{align*}
The latter expression does not depend on $x$ which was arbitrarily chosen outside of $I_d$ and $J_d$. We deduce that
\begin{align*}
&\Vert f-f_{\regparam_d,d}^\alpha \Vert_{\Leb^1(\X)}\\
\leq&\: \mathrm{vol}(\X) \left( \frac{2\delta_0}{\sqrt{d} - 1} + d^{-1/4} \right) + \mathrm{diam}(\Y) \left(  \mathrm{vol}(I_d) + \mathrm{vol}(J_d) \right)
\end{align*}
and the result follows by invoking Lemma \ref{lem:lemmaId} and using Inequality \eqref{eq:Jd}.
\end{proof}

\section{Numerical examples}\label{sec_numeric}

\subsection{Computational tractability}

Working with a large class of approximation functions may pose computational difficulties. An advantage of our Christoffel-Darboux semi-algebraic approximant is that it can be computed efficiently. 

If the input moments are exactly known and given in rational form, the Christoffel-Darboux polynomial to be partially minimized is obtained through formal inversion of the moment matrix. This operation has efficient implementations, namely polynomial time algorithms over rational entries, an example is given in \cite{bareiss1968sylvester}.

In most of the applications, the moments are however known only approximately, and the Christoffel-Darboux polynomial is constructed via the numerical eigenvalue decomposition of the approximate moment matrix. Since the moment matrix is symmetric, its eigenvalue decomposition can be computed efficiently with numerically stable algorithms in floating point arithmetic
\cite{GoluVanl96}.

In addition, the computational overhead of evaluating the semi-algebraic approximant at a given point $\xx$ is that of minimizing a univariate polynomial over the segment $[-1,1]$. The Lipschitz constant of a univariate polynomial with coefficients $\pp = (p_0,\ldots,p_{2d})$ over $[-1,1]$ is at most $\|\pp\|_1$ and hence grid search finds an  $\epsilon$-accurate solution  to this problem using $\frac{2\|\pp\|_1}{\epsilon}$ evaluations. In our case the entries of $\pp$ are polynomials in $\xx$ which are deduced from moment data  so that for a fixed $d$ estimating and evaluating our semi-algebraic approximant up to a fixed arbitrary precision with rational inputs (moment matrix and $\xx$) can be done in polynomial time. Note also that our analysis shows that a level of precision of the order $d^{-p-2}$ is sufficient so that the cost of the overall procedure has a complexity which is polynomial in the bit size of the moment matrix, the target evaluation point $\xx$, as well as in $d$, the degree bound.

\subsection{Prototype code}

We provide a simple Matlab prototype to validate our algorithm. All the examples described below are reproducible, and the Matlab scripts can be found at
\begin{verbatim}
homepages.laas.fr/henrion/software/momgraph
\end{verbatim}
The calling syntax of the main routine is
\begin{verbatim}
[Y,P] = momgraph(M,X)
\end{verbatim}
It takes as an input an approximate moment matrix
\[
\Mmat = \int_\X \mathbf{b}(\xx,f(\xx))\mathbf{b}(\xx,f(\xx))^\top d\xx
\]
(in Matlab floating point format)
for $\mathbf{b}$ the monomial basis vector (in grevlex ordering), and a collection of points 
\[
\left\{\xx_1, \, \xx_2, \ldots, \xx_N\right\} \subset \X\subset [-1,1]^{p-1}
\]
(in Matlab floating point format) with $p>1$.
It outputs an approximation \[
\left\{y_1, \, y_2 , \ldots, y_N\right\} \subset \Y:=[-1,1]
\]
(in Matlab floating point format) of the values $\left\{f(\xx_1), \,f(\xx_2) ,\ldots, f(\xx_N)\right\}$, as well as a matrix $\mathbf{P}$ of coefficients (in Matlab floating point format) of the vector of polynomials $\mathbf {p}$ whose sum of squares yields the Christoffel-Darboux polynomial.

Our implementation is straightforward, not optimized for efficiency. The regularization parameter $\regparam$ is set to the default value of $10^{-8}$, and the Christoffel-Darboux polynomial is computed from the eigenvalue decomposition (Matlab's command {\tt eig}) of the approximate moment matrix $\Mmat+\regparam\Imat$.

\subsection{Sign function}

Consider the measure supported on the graph of the sign function $f(\xx) = \sign(\xx)$
whose moments in the monomial basis on $\X:=[-1,1]$ are
\[
\int_{-1}^1 \xx^{a_1} f(\xx)^{a_2} d\xx = (-1)^{a_2}\int_{-1}^0 \xx^{a_1}d\xx + (1)^{a_2} \int_0^1 \xx^{a_1}d\xx = \frac{(-1)^{a_2}(0^{a_1+1}-(-1)^{a_1+1})+1-0^{a_1+1}}{a_1+1}
\]
for $(a_1,\:a_2) \in \N^2$.
Here is an example of the use of {\tt momgraph} to recover the sign function from the (floating point approximations) of the (exact) moments:
\begin{verbatim}
>> M  % moment matrix of degree 4 for the sign function

M =

    2.0000         0         0    0.6667    1.0000    2.0000
         0    0.6667    1.0000         0         0         0
         0    1.0000    2.0000         0         0         0
    0.6667         0         0    0.4000    0.5000    0.6667
    1.0000         0         0    0.5000    0.6667    1.0000
    2.0000         0         0    0.6667    1.0000    2.0000

>> X = linspace(-1,1,1e3)'; % samples for evaluation
>> [Y,P] = momgraph(M,X); % Christoffel-Darboux approximation
>> plot(X,graph(X),'-r','linewidth',6); hold on; % exact graph
>> plot(X,Y,'-k','linewidth',3); xlabel('x'); ylabel('y');  % approximate graph
\end{verbatim}
This code corresponds to a degree 4 approximation from a moment matrix of size 6 with 15 moments.
A degree 2 approximation can be obtained from its 3 by 3 submatrix
\begin{verbatim}
>> M(1:3,1:3)

ans =

    2.0000         0         0
         0    0.6667    1.0000
         0    1.0000    2.0000
\end{verbatim}
\begin{figure}
\centering
 \includegraphics[width=.45\textwidth]{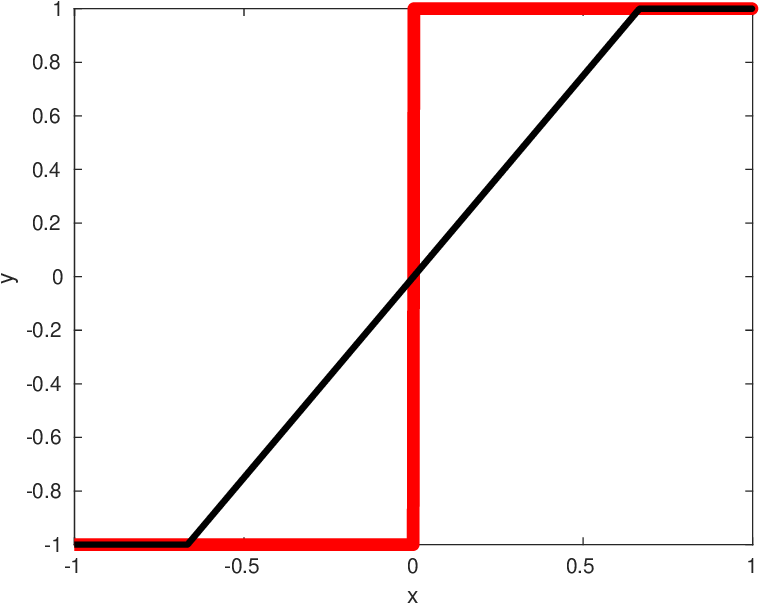} \includegraphics[width=.45\textwidth]{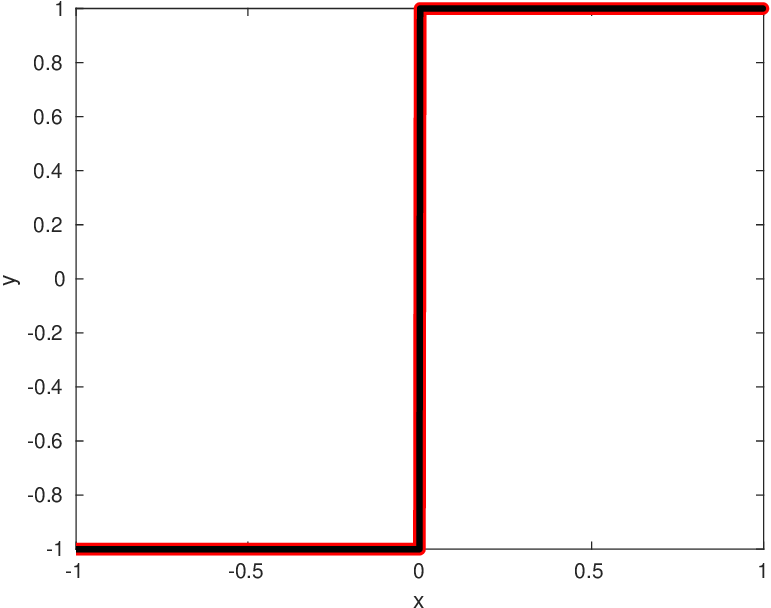}
 \caption{Graph of the sign function (red) and its degree 2 (left) and degree 4 (right) semi-algebraic approximations (black).}
 \label{fig:sign}
 \end{figure}
On Figure \ref{fig:sign} we see that the resulting degree 4 semi-algebraic approximation cannot be distinguished from the graph of the sign function. Our semi-algebraic approximant is $\xx \mapsto \mathrm{argmin}_{y \in [-1,1]} q(\xx,y)$ with $q$ the Christoffel-Darboux polynomial constructed as the sum of squares of the polynomials returned by the {\tt momgraph} function:
\begin{verbatim}
>> mpol x y; b = mmon([x y],2); % GloptiPoly monomial vector of degree 2
>> P*b

6-by-1 polynomial vector

(1,1):7071.0678-7071.0678y^2
(2,1):0.86713+9.4x^2-9.7305xy+0.86713y^2
(3,1):2.4315x-1.3013y
(4,1):-0.53517+1.2757x^2+1.137xy-0.53517y^2
(5,1):0.29635x+0.55374y
(6,1):0.29443+0.10761x^2+0.15643xy+0.29443y^2
\end{verbatim}
We see in particular that the first polynomial is $(\xx,y)\mapsto 1-y^2$ with a large scaling factor. This polynomial vanishes on the graphs of the functions $y \mapsto -1$ and $y \mapsto 1$. The other polynomials are instrumental to determining which one of the two graphs corresponds to a given value of $\xx$.

\subsection{Discontinuous functions}

Let us revisit the discontinuous univariate examples of \cite{eckhoff1993accurate}. Since in this case we do not have the analytic moments of the measure supported on the graph of the function $f$ to input to our algorithm, we use the empirical moment matrix computed by uniform sampling, i.e.
\begin{equation}\label{eq:empirical}
M=\frac{1}{N} \sum_{k=1}^N \mathbf{b}(\xx_k, f(\xx_k))\mathbf{b}(\xx_k, f(\xx_k))^\top
\end{equation}
for $N$ sufficiently large, i.e. $10^3$, and $\mathbf{b}$ the monomial basis vector. 
Degree 10 semi-algebraic approximations are reported on Figure \ref{fig:eckhoff} for three benchmarks \cite[Examples 65, 66, 67]{eckhoff1993accurate}
of discontinuous functions $f$, appropriately scaled in $\X=\Y=[-1,1]$.
\begin{figure}
\centering
\includegraphics[width=.3\textwidth]{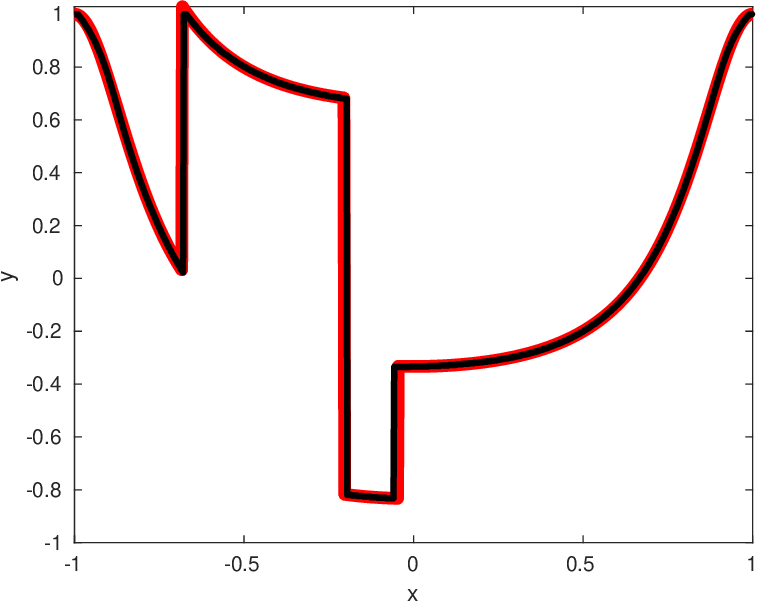} 
\includegraphics[width=.3\textwidth]{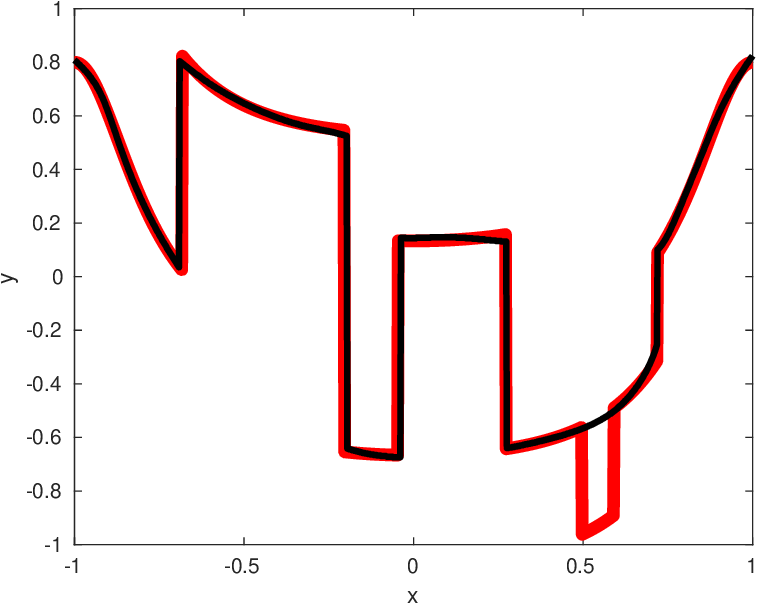}
\includegraphics[width=.3\textwidth]{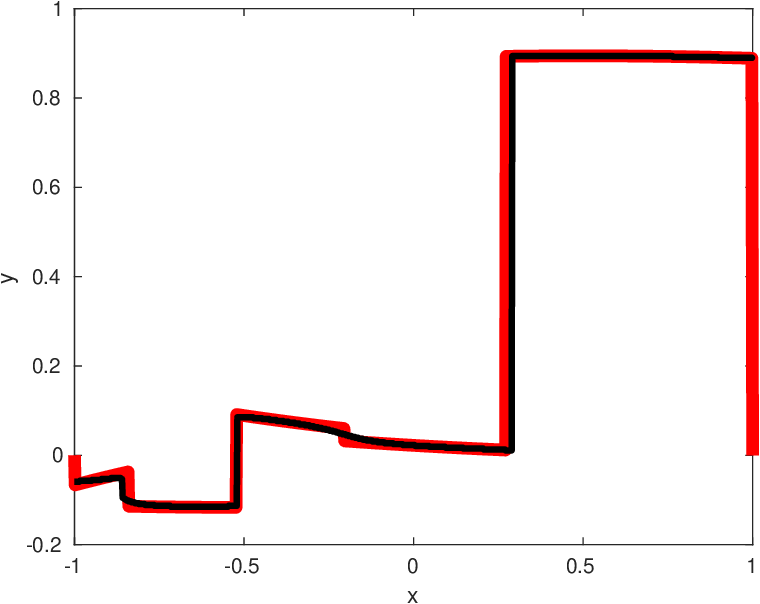}
 \caption{Degree 10 semi-algebraic approximations (black) for the discontinuous univariate functions (red) of Examples 65 (left), 66 (middle) and 67 (right) of reference \cite{eckhoff1993accurate}.}
 \label{fig:eckhoff}
 \end{figure}
We observe that the second rightmost discontinuity in the middle example is not detected. Increasing the degree of the approximations does not fix the issue, and we believe that it is due to the poor resolution of the monomial basis. It would be more appropriate to use here a complex exponential basis (i.e. Fourier coefficients) or an orthogonal basis (e.g. Chebyshev or Legendre polynomials).

\subsection{Interpolation}

\begin{figure}
\centering
\includegraphics[width=.45\textwidth]{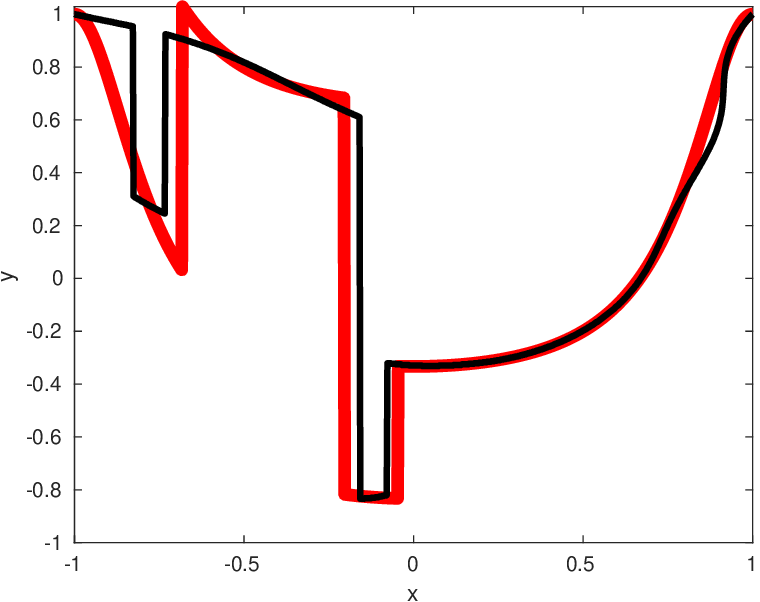} 
\includegraphics[width=.45\textwidth]{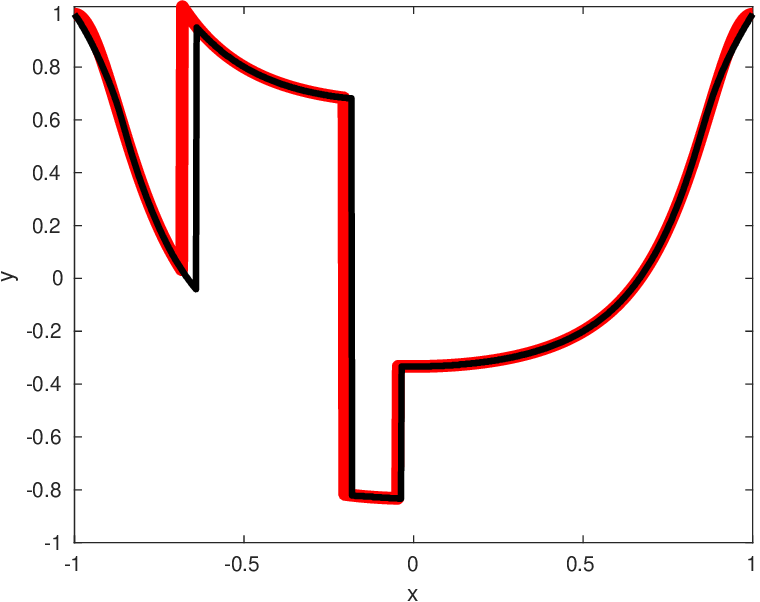}
\includegraphics[width=.45\textwidth]{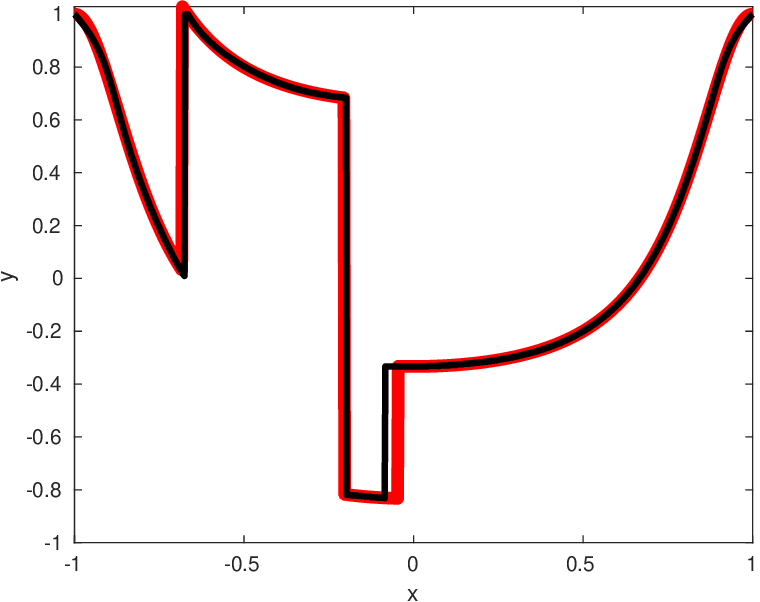}
\includegraphics[width=.45\textwidth]{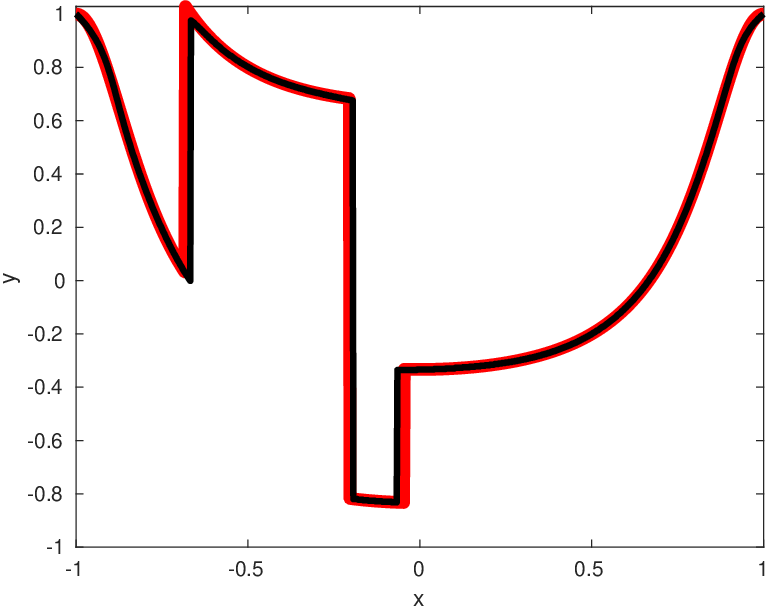}
 \caption{Degree 10 semi-algebraic approximations (black) of a discontinuous function (red) computed from empirical moments evaluated at 10 (upper left), 20 (upper right), 30 (lower left) and 40 (lower right) uniformly distributed samples.}
 \label{fig:eckhoff_sampled}
 \end{figure}
Suppose now that we have access only to the values $\{f(\xx_k)\}_{k=1,\ldots,N}$ of the function to be approximated at given sampling points $\{\xx_k\}_{k=1,\ldots,N}$, for $N$ small. Our algorithm takes as input the empirical moment matrix \eqref{eq:empirical}.
On Figure \ref{fig:eckhoff_sampled}, we revisit \cite[Example 65]{eckhoff1993accurate} to study the effect of the number of samples $N$ on the quality of the approximation, for a uniform distribution of samples.
We see that with 20 samples the function is already well approximated.

\subsection{Recovering trajectories for optimal control}

In \cite{lasserre2008nonlinear}, the moment-SOS hierarchy is applied to solve numerically non-linear optimal control of ODEs with polynomial data and semi-algebraic state and control constraints. Non-linear optimal control is formulated as a linear problem on moments of occupation measures supported on optimal trajectories. Let us show how numerical approximations of these moments obtained by semidefinite programming can be input to our algorithm to approximate optimal state and control trajectories.

Let us revisit the state-constrained double integrator problem of \cite[Section 5.1]{lasserre2008nonlinear} to approximate the time optimal trajectories. After a scaling of time, state and control, we use the Matlab interface GloptiPoly 3 and the conic solver MOSEK to compute the pseudo-moments of the occupation measure of degree up to 8. This can be achieved in less than 2 seconds on a standard desktop computer. From this output, we construct the 45-by-45 moment matrices of the control and state marginals, conditioned w.r.t. time. Using our notations, the independent variable $\mathbf{x}$ is time, while the dependent variable $y$ is respectively the control, the first state and the second state. For this example, the analytic trajectories are available for comparison.
\begin{figure}
\centering
\includegraphics[width=.3\textwidth]{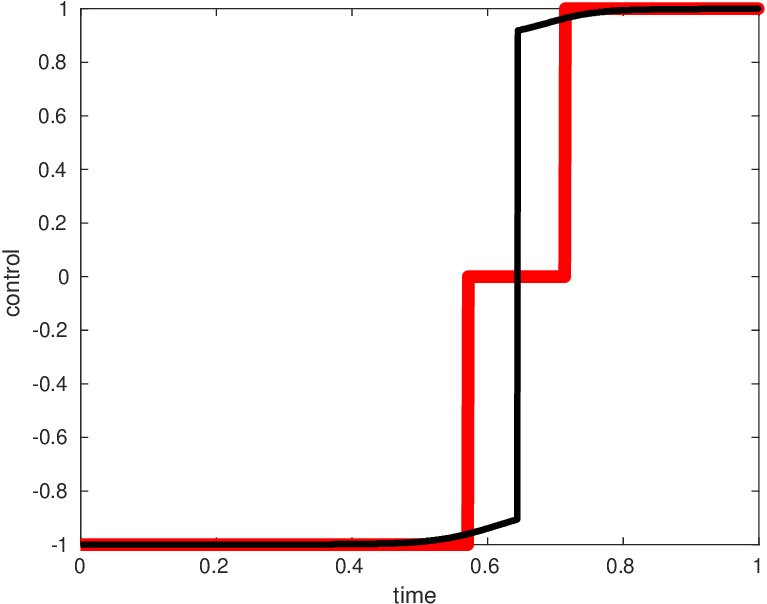} 
\includegraphics[width=.3\textwidth]{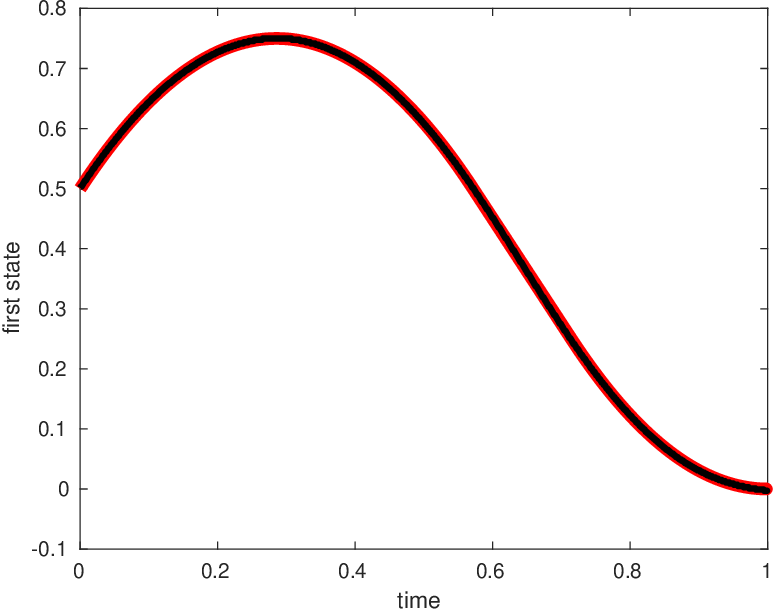}
\includegraphics[width=.3\textwidth]{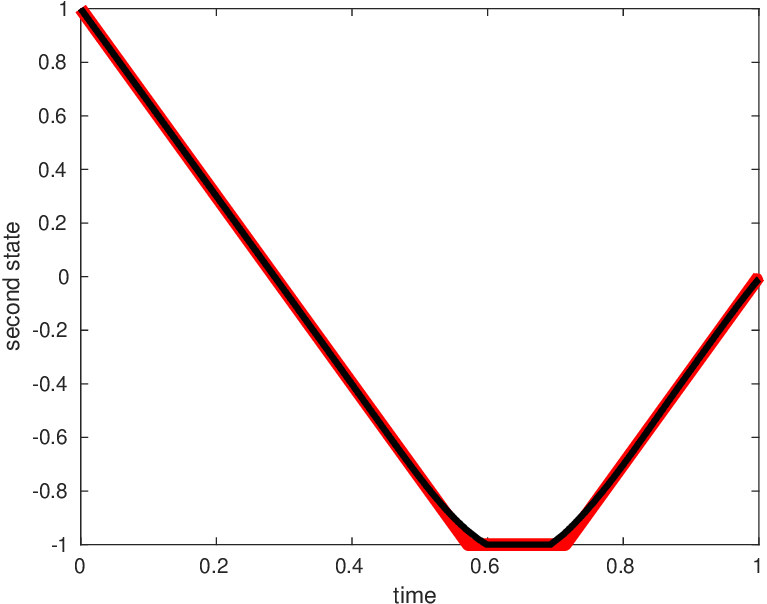}
 \caption{Minimum time double integrator with state constraint: control (left), first state (middle) and second state (right) trajectories (red) and their degree 8 semi-algebraic approximations (black) constructed from the pseudo-moments of the occupation measure.}
 \label{fig:ocp}
 \end{figure}
We see on Figure \ref{fig:ocp} that the state trajectory approximations are tight, whereas the control trajectory approximation misses partly the central region corresponding to the saturation of the second state. Indeed, since it is obtained by solving numerically the degree 8 semidefinite relaxation of the moment-SOS hierarchy, the approximated moment matrix differs from the exact moment matrix, and this has an impact on the quality of the Christoffel-Darboux approximation.
\begin{figure}
\centering
\includegraphics[width=.3\textwidth]{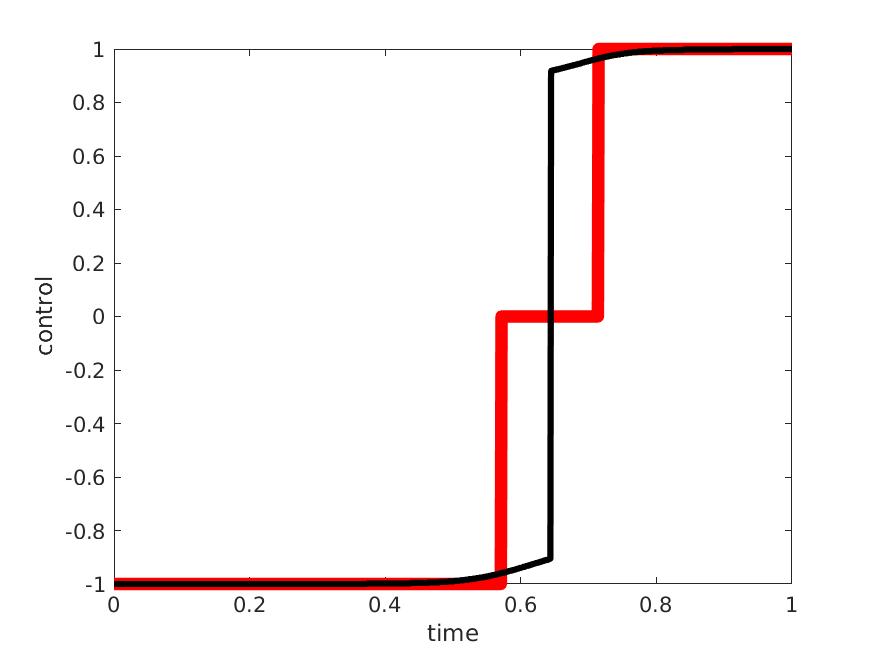} 
\includegraphics[width=.3\textwidth]{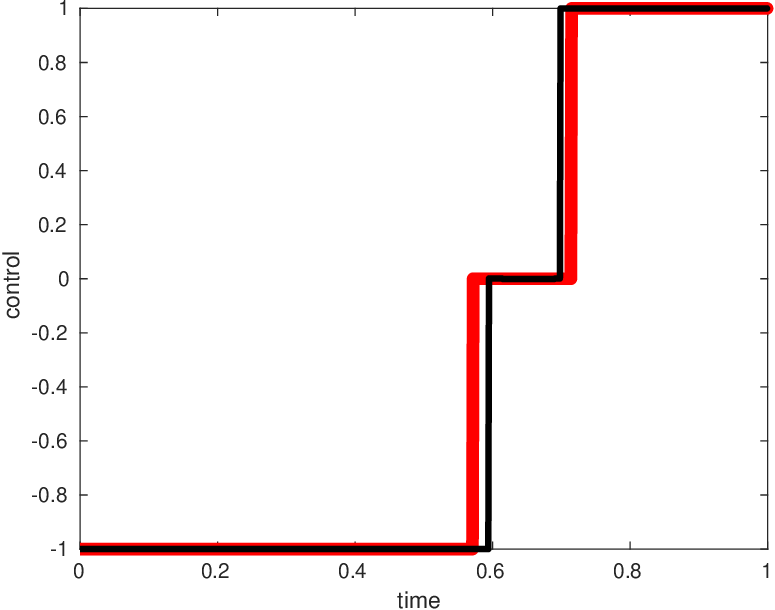}
 \caption{Minimum time double integrator with state constraint: control trajectories (red) and their degree 8 semi-algebraic approximations (black) constructed from the pseudo-moments (left, same as left of Figure \ref{fig:ocp}) and from the analytic moments (right)  of the occupation measure.}
 \label{fig:ocpa}
 \end{figure}
For this example, we can construct analytically the exact moment matrix of the control trajectory and observe that indeed its Christoffel-Darboux semi-algebraic approximation of degree 8 identifies well the optimal control trajectory switching times, see Figure \ref{fig:ocpa}.

\subsection{Bivariate examples}

\begin{figure}
\centering
\includegraphics[width=.45\textwidth]{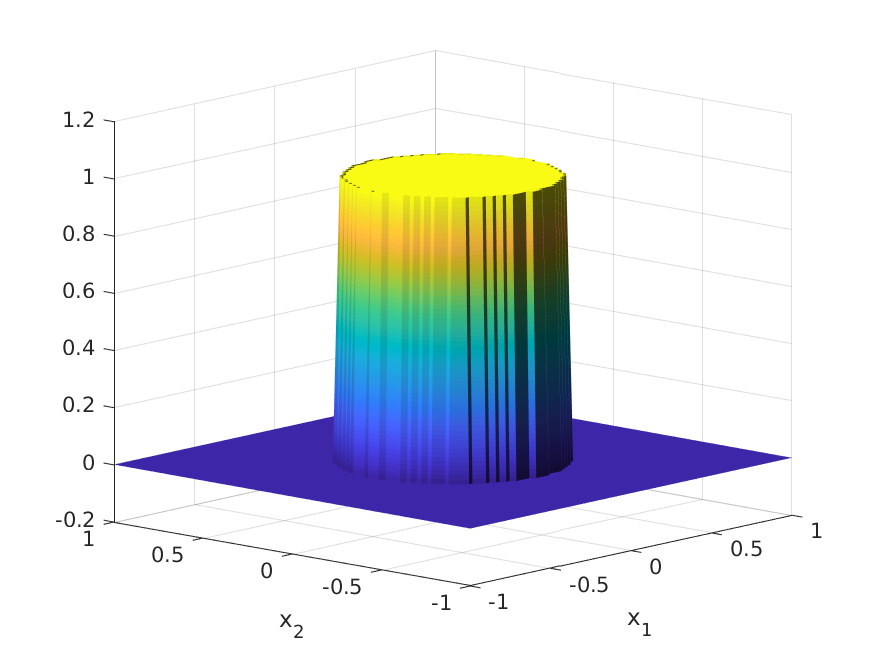} 
\includegraphics[width=.45\textwidth]{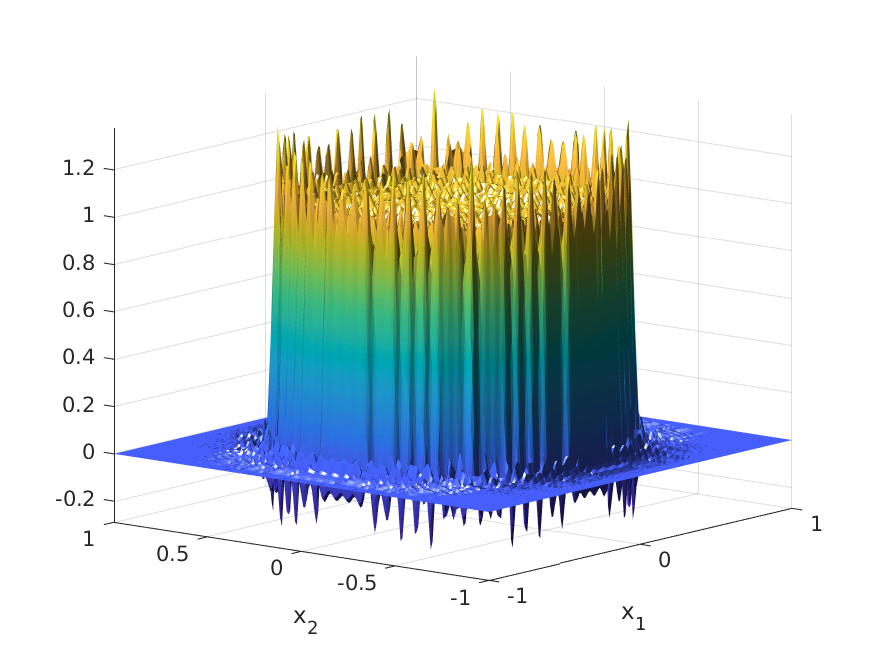}
 \caption{Degree 4 (left) semi-algebraic approximation, and Chebyshev polynomial approximation (right) of the indicator function of a disk.}
 \label{fig:onedisk}
 \end{figure}
Consider the indicator function
\[
f(\xx) :=\mathbb{I}_{\{\xx \in \R^2 : \xx_1^2+\xx_ 2^2 \leq 1/4\}}(\xx)
\]
of a centered disk of radius $1/2$. We compute the emprical moments obtained by sampling $100^2$ points on a uniform grid of $\X:=[-1,1]^2$. With this input, our algorithm computes the degree 8 semi-algebraic approximation reported on Figure \ref{fig:onedisk}, to be compared with the Chebyshev polynomial approximation obtained from $100^2$ points by the {\tt chebfun2} command, showing the typical Gibbs phenomenon.
\begin{figure}
\centering
\includegraphics[width=.45\textwidth]{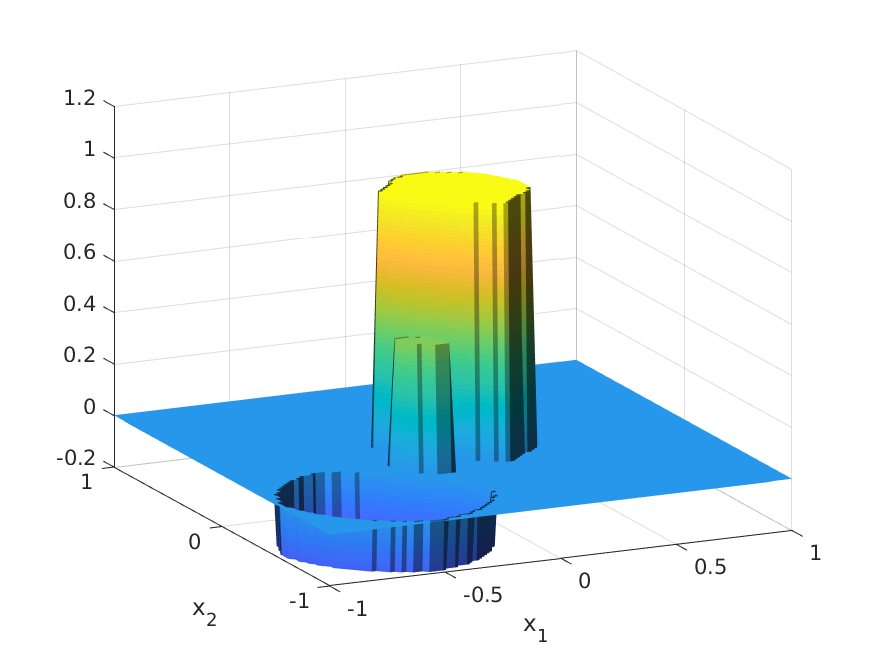} 
\includegraphics[width=.45\textwidth]{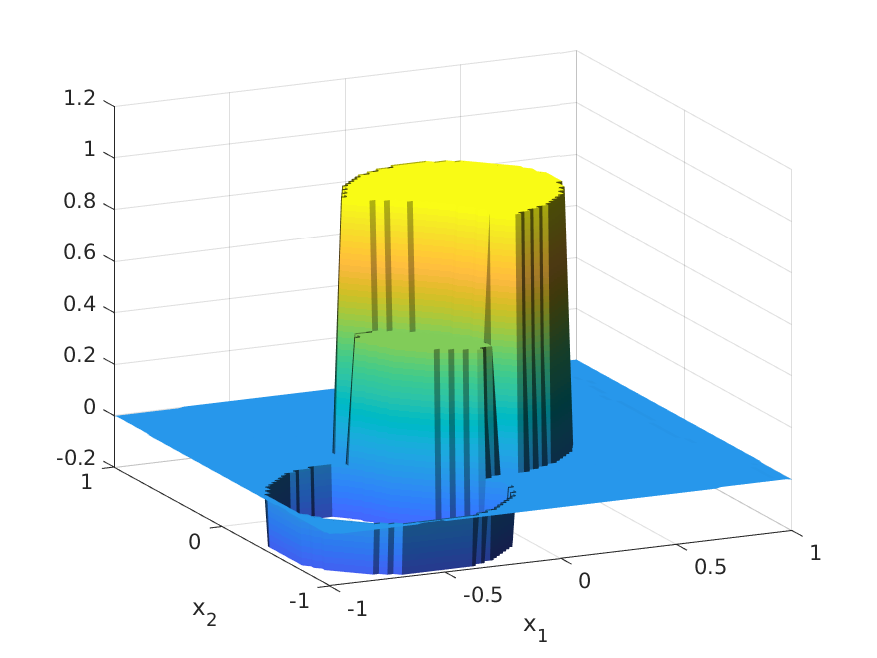}
 \caption{Degree 8 (left) and degree 16 (right) semi-algebraic approximations of the superposition of signed indicator functions of two disks.}
 \label{fig:twodisks}
 \end{figure}
 \begin{figure}
\centering
\includegraphics[width=.45\textwidth]{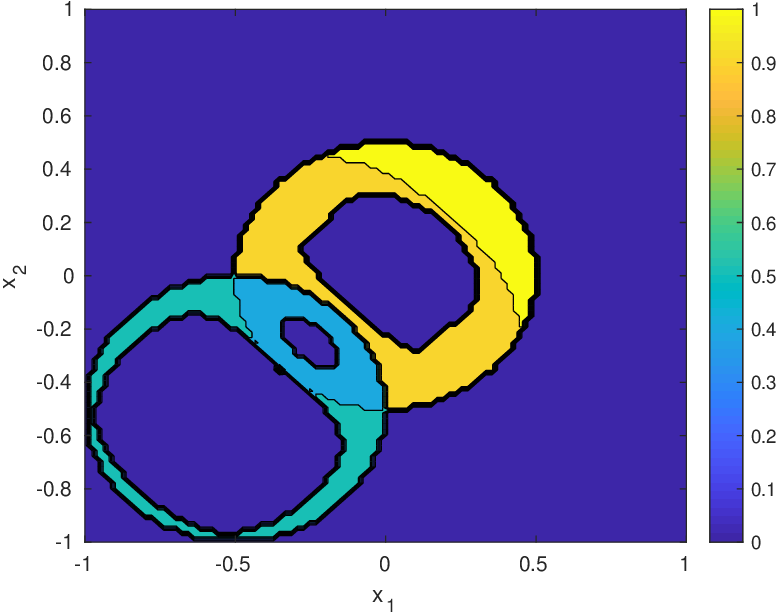} 
\includegraphics[width=.45\textwidth]{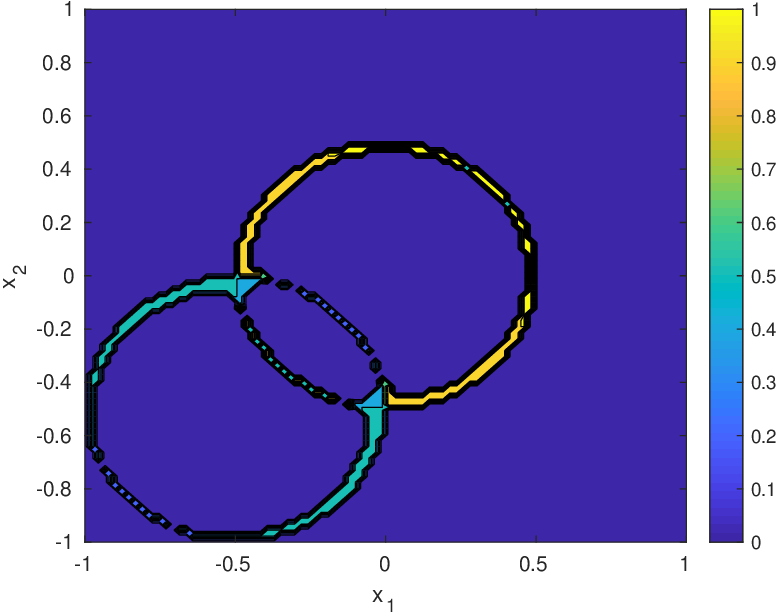}
 \caption{Contour plots of the absolute error between the two disk indicator function and its degree 8 (left) and degree 16 (right) semi-algebraic approximations of Figure \ref{fig:twodisks}.}
 \label{fig:twodiskserror}
 \end{figure}
 
We perform the same computations for the piecewise constant function
\[
f(\xx) :=\mathbb{I}_{\{\xx \in \R^2 : \xx_1^2+\xx_ 2^2 \leq 1/4\}}(\xx)-\frac{1}{2}\mathbb{I}_{\{\xx \in \R^2 : (\xx_1+\frac{1}{2})^2+(\xx_ 2+\frac{1}{2})^2 \leq 1/4\}}(\xx)
\]
obtained as a superposition of signed indicator functions of two disks. Its degree 8 and 16 semi-algebraic approximations are reported on Figure \ref{fig:twodisks}. The absolute pointwise error between the approximations and the original function is displayed on Figure \ref{fig:twodiskserror}.
 
\subsection{Discontinuous solutions of non-linear PDEs}

 \begin{figure}
 \centering
 \includegraphics[width=.75\textwidth]{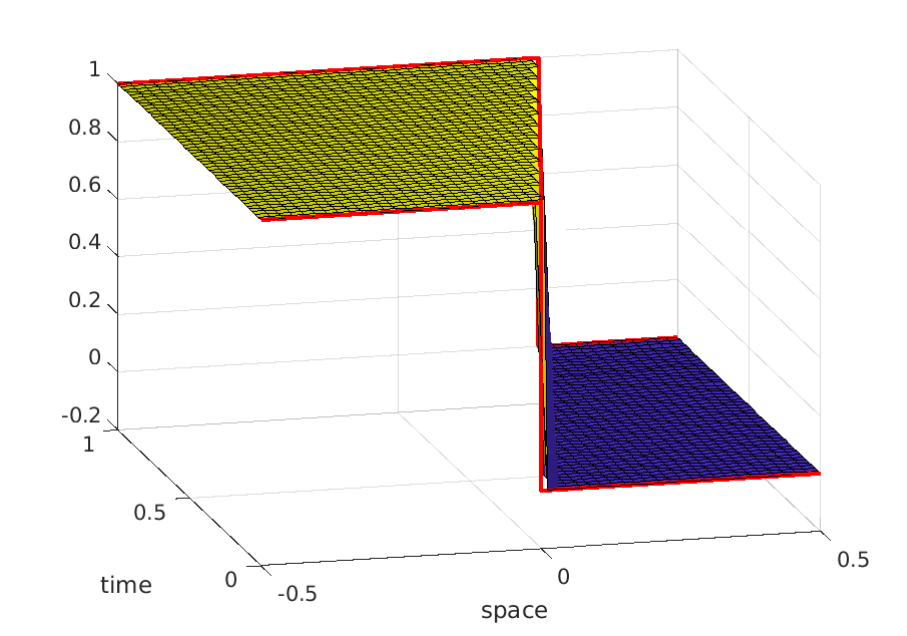}
 \caption{Graph of the solution (a function of time and space) recovered from approximate moments for the Burgers PDE: Discontinuous initial data. The shock propagates linearly with time.}
 \label{fig:burgers-shock-initialtime}
 \end{figure}

\begin{figure}
\centering
\includegraphics[width=.75\textwidth]{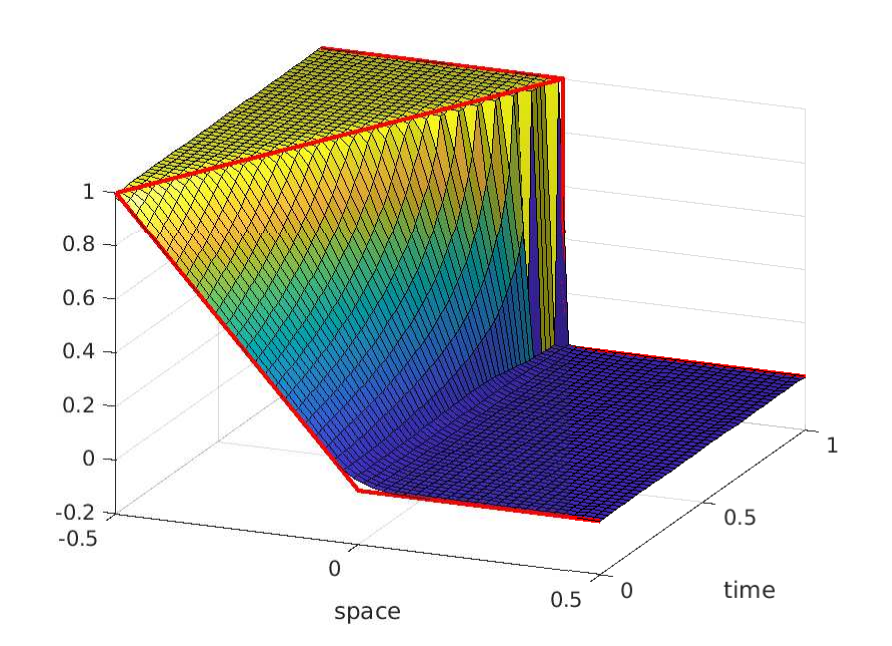} 
\caption{Graph of the solution (a function of time and space) recovered from approximate moments for the Burgers PDE: Initial condition chosen to produce a shock at final time.}
\label{fig:burgers-shock-finaltime}
\end{figure}

  \begin{figure}
  \centering
  \includegraphics[width=.75\textwidth]{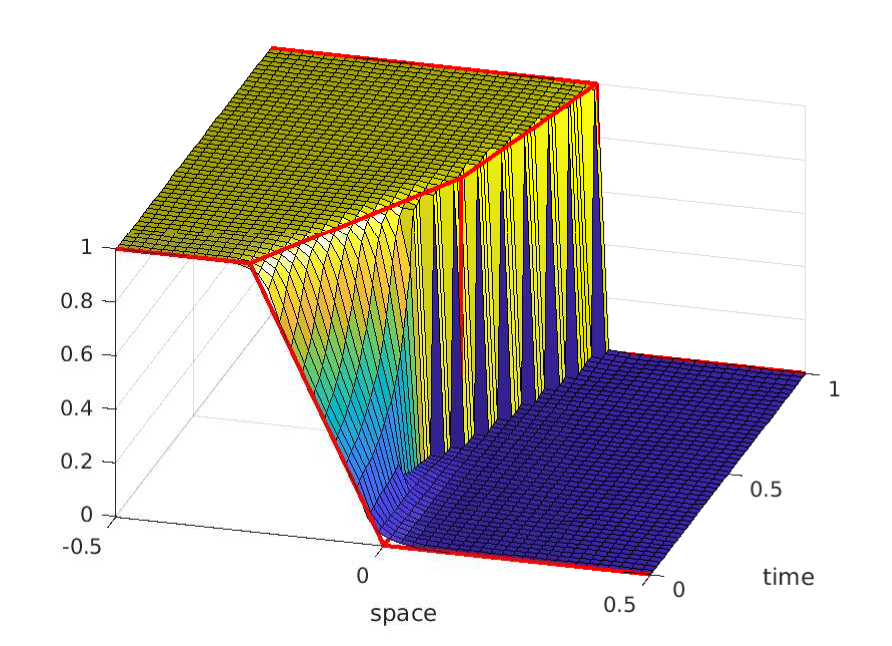}  
  \caption{Graph of the solution (a function of time and space) recovered from approximate moments for the Burgers PDE: Initial function chosen such that the shock occurs at $t=\frac{1}{2}$.}
  \label{fig:burgers-shock-midtime}
  \end{figure}

In \cite{marx2018moment}, the moment-SOS hierarchy is applied to solve numerically a class of non-linear PDEs for which we known that classical (i.e. differentiable) solutions do not exist. The advantage of optimizing over occupation measures is that they can be supported on graphs of weak (i.e. possibly discontinuous) solutions. Let us show how approximate moments of these measures computed by semidefinite programming can be processed by our algorithm so as to recover these discontinuous solutions. 

We focus on the Burgers equation and choose the initial data (a function of one space coordinate, at time zero) in a way that at a given time a shock appears, i.e. the solution becomes a discontinuous function of the space coordinate. Once the shock appeared, it propagates through, i.e. the discontinuity remains but its location varies. 
In Figures \ref{fig:burgers-shock-initialtime}, \ref{fig:burgers-shock-finaltime}, and \ref{fig:burgers-shock-midtime} we show the graphs obtained from the moment relaxations proposed in \cite{marx2018moment}. In all cases we use the 969 triviate moments of degree 16 of the occupation measure (supported on time, space, and solution) to recover the graph of the approximated solution. For comparison we also sketch the analytic solution with red lines. 

For the graphs in Figures \ref{fig:burgers-shock-initialtime} and \ref{fig:burgers-shock-finaltime}, the approximated moments match the analytic moments up to an error of the order of $10^{-8}$. Our semi-algebraic approximations are almost identical to the analytic solution.

For the graph in Figure \ref{fig:burgers-shock-midtime} the approximated moments are noticeably incorrect, i.e., the error is of order $10^{-4}$. Nevertheless, our semi-algebraic approximation is able to reproduce the graph of the solution quite accurately. In particular the propagation of the shock is retrieved from the moment data. However, the approximation is erroneous when the solution passes over from its continuous to its discontinuous part. 

\section{Conclusion}

\label{sec_conclusion}

In this paper, we describe a new technique to estimate discontinuous functions from moment data, based on  Christoffel-Darboux  kernels. Instead of using polynomial or piecewise polynomial approximants, we use a class of semi-algebraic approximants, namely arguments of minima of polynomials. This is another occurrence of a lifting technique: instead of using only moments depending linearly on the function so as to recover directly the function, we use also moments depending non-linearly on the function so as to approximate the support of a measure concentrated on the graph of the function. We provide functional analytic and geometric convergence proofs. Finally, some numerical examples illustrate the efficiency of our algorithm.

We believe that this work opens the way to many other further research lines:
\begin{itemize}
\item When applying the Moment-SOS hierarchy, the moments are numerical approximations of the real ones. It would be interesting to provide a sensibility analysis of the application of our algorithm for the real moments and the approximated ones. We believe that such an analysis can be performed, since promising results were achieved recently in \cite{klep2018minimizer} for the case of zero dimensional manifolds, i.e. unions of finitely many points.
\item It could also be interesting to investigate in a more quantitative way why the Gibbs phenomenon might be avoided or at least attenuated with the technique we provide. We believe that it is mainly due to the semi-algebraic point of view we are following. 
\item We could also check whether our algorithm works as well when considering only the knowledge of Fourier coefficients, namely moments depending linearly on the function that we want to approximate. In many problems, this is the only measurement that we might have. This is therefore a partial moment information, and we may want to complement it with estimates of higher degree moments. This makes the problem challenging.
\end{itemize}

\section{Acknowledgments}

We are grateful to Quentin Vila for his technical input on discontinuous solutions of PDEs, and to Milan Korda, Victor Magron and Matteo Tacchi for interesting discussions. This work was partly funded by the ERC Advanced Grant Taming and was also conducted in the framework of the regional programme ”Atlanstic 2020, Research, Education and Innovation in Pays de la Loire, supported by the French Region Pays de la Loire and the European Regional Development Fund. E. Pauwels and J.B. Lasserre are also partially supported by the AI Interdisciplinary Institute ANITI funding through the french program “Investing for the Future PIA3”, under the Grant agreement number ANR-19-PI3A-0004.

\bibliographystyle{plain}

\end{document}